\documentclass[a4paper,11pt]{amsart}
 \usepackage[utf8x]{inputenc}
 \usepackage{amsthm,amsfonts,amsmath,amssymb,enumerate}
 \usepackage{verbatim} % comentarios
 \usepackage{rotating} % rotar tablas

 \usepackage{pdfsync}
 \usepackage{cancel}
 
 \usepackage{multirow}
 \usepackage{color}
 
 \usepackage{tikz}
 \usetikzlibrary{calc}
 \usetikzlibrary{decorations.pathreplacing}
 \usepackage[all]{xy}

 \newtheorem{theorem}{Theorem}[section]
 \newtheorem{lemma}[theorem]{Lemma}
 \newtheorem{corollary}[theorem]{Corollary}
 
 \newtheorem{proposition}[theorem]{Proposition}

 \newtheorem*{thm}{Theorem}
 
 \theoremstyle{definition}
 \newtheorem{definition}[theorem]{Definition}
 \newtheorem{example}[theorem]{Example}

 \theoremstyle{remark}
 \newtheorem{remark}[theorem]{Remark}

 \numberwithin{equation}{section}
 
 \newcommand{\g}{\mathfrak{g}}

 \newcommand{\N}{\mathbb{N}}
 \newcommand{\K}{\mathbb{K}}

 \newcommand{\F}{\mathcal{F}}
 
 \newcommand{\mult}{\mathtt{m}}
 
\begin{document}
 	
\title[$k$-step nilpotent symplectic Lie algebras]{$k$-step nilpotent symplectic Lie algebras associated with graphs}

\author{Josefina Barrionuevo, Paulo Tirao and Sonia Vera}

\subjclass{53D05, 53D10 (primary), 17B56 (secondary)}

\date{April 2026}

\begin{abstract}
We construct families of $k$-step nilpotent symplectic Lie algebras associated with graphs, 
extending the construction given in \cite{PT} for the 2-step case. 

We also show that, under mild conditions on the nilpotency type, there exist symplectic Lie algebras of that type.
\end{abstract}

\maketitle

%==============================================================
\setcounter{section}{-1}
\section{Introduction}

Symplectic geometry is a central area of mathematics, with deep connections to physics---most notably as the geometric language of classical mechanics---as well as to differential geometry and topology. Its algebraic counterpart is given by the study of \textbf{symplectic Lie algebras}, which arise as the infinitesimal counterparts of symplectic symmetries on Lie groups. These structures provide a rich framework with important connections to representation theory, Poisson geometry, and mathematical physics.

Following the seminal work of Benson and Gordon \cite{BG}, particular attention has been devoted to nilmanifolds, that is, compact quotients of nilpotent Lie groups. Later, Dani and Mainkar \cite{DM} studied the family of \(2\)-step nilpotent Lie algebras associated with simple graphs, together with the corresponding nilmanifolds, drawing particular attention to this class. In particular, \cite{PT} determines which of these graph-associated Lie algebras admit a symplectic form.

The general problem of deciding which nilpotent Lie algebras admit a symplectic form has been studied by many authors. Most available results concern 2-step nilpotent Lie algebras, low-dimensional cases, and filiform Lie algebras. In contrast, examples of $k$-step nilpotent symplectic Lie algebras remain scarce for $k \ge 3$.

Recently, \cite{ABDGH} characterized nilpotent Lie algebras with a codimension-one abelian ideal admitting a symplectic structure. This family includes $k$-step nilpotent Lie algebras for every $k$.

Given a simple graph $G=(V,E)$, one can associate to it a $k$-step nilpotent Lie algebra $\g(k,G)$. If $G$ is the complete graph on $m$ vertices, then $\g(k,G)$ is the free $k$-step nilpotent Lie algebra on $m$ generators. It was shown in \cite{D} that this Lie algebra admits a symplectic form if and only if $k=2$ and $m=3$.

In this paper, we extend the construction introduced in \cite{PT} for 2-step nilpotent symplectic Lie algebras 
associated with graphs to the $k$-step case. 
The nilpotent Lie algebras we construct are all naturally graded. 

All Lie algebras in this paper are finite-dimensional over a a field $\K$ of characteristic zero.

%==============================================================
\section{Some basic preliminaries}

%---------------------------------------
\subsection{Symplectic Lie algebras}\label{subsec:sym}

Let $\g$ be a Lie algebra of dimension $2n$.
A symplectic form on $\g$ is a closed $2$-form $\omega\in \Lambda^2 \g^*$ such that $\omega^n\ne 0$.
If $\{x_1,\dots,x_{2n}\}$ is a basis of $\g$ and $\{x^*_1,\dots,x^*_{2n}\}$ is its dual basis, then a closed $2$-form
\[
\omega=\sum_{i<j} a_{i,j}\, x^*_i \wedge x^*_j
\]
is symplectic if and only if
\[
\omega^n = c\, x^*_1 \wedge x^*_2 \wedge \cdots \wedge x^*_{2n},
\]
for some $c\ne 0$.

\begin{remark}\label{rmk:sym1}
If the indices satisfy
\[
\{i_k:\, k=1,\dots,n\}\cup \{j_k:\, k=1,\dots,n\}=\{1,\dots,2n\},
\]
then
\begin{equation}\label{eqn:sym-form}
\omega = \sum_{k=1}^n a_k\, x^*_{i_k}\wedge x^*_{j_k}
\end{equation}
satisfies $\omega^n\ne 0$, provided that $a_k\ne 0$ for all $k=1,\dots,n$.
\end{remark}

Recall that $\omega$ is closed if $d(\omega)=0$, where $d:\Lambda^2\g^* \to \Lambda^3\g^*$ is the linear map defined by
\[
d(x^*_i\wedge x^*_j)=d^1(x^*_i)\wedge x^*_j-x^*_i\wedge d^1(x^*_j),
\]
and $d^1:\g^* \to \Lambda^2\g^*$ is the dual map of the bracket
\[
[\;,\;]:\Lambda^2\g \to \g.
\]

A Lie algebra admitting a symplectic form is called a symplectic Lie algebra.

%--------------------------------------------------------
\subsection{$k$-step nilpotent Lie algebras}

The lower central series of a Lie algebra $\g$ is the descending filtration of ideals
\[
\g^1\supseteq \g^2 \supseteq \cdots \supseteq \g^k \supseteq \cdots
\]
defined recursively by
\[
\g^1=\g \qquad\text{and}\qquad \g^{i+1}=[\g,\g^i], \quad \text{for } i\ge 1.
\]
Note that it satisfies
\[
[\g^i,\g^j]\subseteq \g^{i+j},
\]
for all \(i,j\ge 1\).

The associated graded Lie algebra is
\[
\operatorname{gr}(\g)= \bigoplus_{i\ge 1} \g^i/\g^{i+1}.
\]
Its Lie bracket is defined by
\[
[\overline{x},\overline{y}] = \overline{[x,y]},
\]
for $x\in\g^i$ and $y\in\g^j$, where $\overline{x}$ and $\overline{y}$ denote the classes of $x$ and $y$ in
$\g^i/\g^{i+1}$ and $\g^j/\g^{j+1}$, respectively, and $\overline{[x,y]}$ denotes the class of $[x,y]$ in $\g^{i+j}/\g^{i+j+1}$.

The Lie algebra $\g$ is said to be $k$-step nilpotent if $\g^{k+1}=0$ and $\g^{k}\ne 0$. It is said to be \emph{naturally graded} if $\g\simeq \operatorname{gr}(\g)$.

A Lie algebra $\g$ with a graded decomposition
\begin{equation}\label{eqn:graded}
	\g=V_1\oplus\dots\oplus V_k,
\end{equation}
such that $[V_i,V_j]\subseteq V_{i+j}$, is at most $k$-step nilpotent. In general, such a Lie algebra may admit different graded decompositions.

If, in addition,
\begin{equation}\label{eqn:carnot}
	\g^i=V_i\oplus\dots\oplus V_k,
\end{equation}
then $\g$ is $k$-step nilpotent and naturally graded. Such a decomposition is called a \emph{Carnot grading}. It is unique up to isomorphism.

\begin{remark}
	Having a Carnot grading and being naturally graded are equivalent notions.
\end{remark}

%----------------------------------------------------------------------------------

\subsection{(Di)graphs}\label{subsec:dig}

A (simple) graph is a pair \(G=(V,E)\), where \(E\subseteq \mathcal{P}^2(V)\).

A directed graph, or \emph{digraph}, is a graph in which each edge is assigned a direction; that is, each edge has an origin and a terminus. Thus, a digraph is a pair \(G=(V,E)\), where \(E\subseteq V\times V\) and, whenever \((u,v)\in E\), one has \((v,u)\notin E\).

By forgetting the directions, every digraph naturally determines a graph.

A \emph{directed path} in a digraph from a source vertex $u$ to a target vertex $v$ is a finite sequence of distinct vertices $v_1,v_2,\dots,v_k$ such that:
\begin{itemize}
   \item $v_1=u$ and $v_k=v$;
   \item $(v_i,v_{i+1})\in E$ for all $i\in\{1,\dots,k-1\}$.
\end{itemize}
Thus, in a directed path, one must always follow the orientation of the edges.

In this paper, we will repeatedly use the following two families of digraphs:
\begin{enumerate}
\item \textbf{A tree directed from the leaves to the root.}
Given a tree and a choice of a root \(v\), there is a unique way to orient its edges so that, for every vertex \(u\neq v\), there is a directed path from \(u\) to \(v\).

\item \textbf{A unicyclic graph directed centripetally.}
Given a unicyclic graph (that is, a graph with exactly one cycle), one may view it as a cycle with a tree attached to some vertices of the cycle. 
Each such tree is oriented as in (1) by taking its cycle vertex as the root, and the cycle itself may be given either of its two possible cyclic orientations.
\end{enumerate}

%---------------------------------------------------------------------
\subsection{2-step nilpotent Lie algebras associated with graphs and 2-forms}

Let $G=(V,E)$ be a graph, with $V=\{v_1,\dots,v_m\}$.
The associated 2-step nilpotent Lie algebra is the quotient  
\[ \g(2,G)=L_{(2)}(m)/I, \]
where  $L_{(2)}(m)$ is the free 2-step nilpotent Lie algebra and $I$ is the ideal
generated by $\{[v_i, v_j]:\; \{v_i,v_j\}\not\in E \}$.

A basis of $\g(2,G)$ is 
\[\{ v_1,\dots,v_m\}\cup \{[v_i,v_j] : \{v_i,v_j\}\in E, i<j\}, \]
 and its corresponding dual basis of $\g(2,G)^*$ is
\[ \{ v_1^*,\dots,v_m^*\}\cup \{[v_i,v_j]^*: \{v_i,v_j\}\in E, i<j\}. \]

Note that, for these basis vectors, $d^1(v_i^*)=0$ and $d^1([v_i,v_j]^*)=v_i^*\wedge v_j^*$,
and hence the $2$-forms
\begin{equation}\label{eqn:sigma-ij}
	v_i^* \wedge v_j^* \qquad\text{and}\qquad \sigma_{i,j}=v_i^*\wedge[v_i,v_j]^*
    \end{equation}
are closed.

%==============================================================
\section{2-step nilpotent graph symplectic Lie algebras}\label{sec:2-gr}

In this section, for the graphs that yield a 2-step nilpotent symplectic Lie algebra, as characterized in \cite{PT}, we exhibit a symplectic form arising from a particular choice of digraph structure on the graph.

\begin{thm}[\cite{PT}]
The 2-step Lie algebra associated with a graph $G=(V,E)$ is symplectic if and only if
$|V|+|E|$ is even and, in each connected component of $G$, the number of edges does not exceed the number of vertices.
\end{thm}

If a graph $G$ satisfies the hypotheses of the theorem, then each connected component has at most one cycle. Moreover, each acyclic component must be paired with another acyclic component, since $|V|+|E|$ is even. Therefore, any graph satisfying the hypotheses of the theorem is a disjoint union of connected components of the following two types:
\begin{itemize}
	\item trees, taken in pairs;
	\item unicyclic graphs.
\end{itemize}

For each 2-step nilpotent Lie algebra associated with a graph in one of these classes, there is a symplectic form defined using the orientations introduced in Section~\ref{subsec:dig}.

Let $T$ be a tree with a fixed root $v_1$. We consider the digraph obtained by orienting $T$ from the leaves toward the root (see Section~\ref{subsec:dig}).

For example, for the rooted tree $T$:

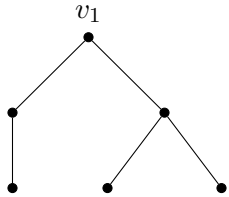
\begin{center}
	\begin{tikzpicture}
		\tikzset{every loop/.style={looseness=30}}
		\node (1) at (0,-0) [draw,circle,inner sep=1.2pt,fill=black!100,label=90:{$v_1$}] {};
		\node (2) at (-1,-1)  [draw,circle,inner sep=1.2pt,fill=black!100,label=180:{}] {};
		\node (3) at (1,-1)  [draw,circle,inner sep=1.2pt,fill=black!100,label=0:{}] {};
		\node (4) at (-1,-2)  [draw,circle,inner sep=1.2pt,fill=black!100,label=180:{}] {};
		\node (5) at (0.25,-2)  [draw,circle,inner sep=1.2pt,fill=black!100,label=270:{}] {};
		\node (6) at (1.75,-2)  [draw,circle,inner sep=1.2pt,fill=black!100,label=0:{}] {};
		
		\path
		(1) edge [-, black] node[above=0] {} (2)
        (1) edge [-, black] node[above=0] {} (3)
        (2) edge [-, black] node[above=0] {} (4)
        (3) edge [-, black] node[above=0] {} (5)        
        (3) edge [-, black] node[above=0] {} (6);
	\end{tikzpicture}
\end{center}
the corresponding orientation is

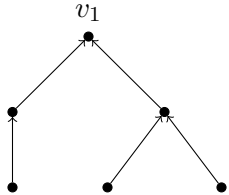
\begin{center}
	\begin{tikzpicture}
		\tikzset{every loop/.style={looseness=30}}
		\node (1) at (0,-0) [draw,circle,inner sep=1.2pt,fill=black!100,label=90:{$v_1$}] {};
		\node (2) at (-1,-1)  [draw,circle,inner sep=1.2pt,fill=black!100,label=180:{}] {};
		\node (3) at (1,-1)  [draw,circle,inner sep=1.2pt,fill=black!100,label=0:{}] {};
		\node (4) at (-1,-2)  [draw,circle,inner sep=1.2pt,fill=black!100,label=180:{}] {};
		\node (5) at (0.25,-2)  [draw,circle,inner sep=1.2pt,fill=black!100,label=270:{}] {};
		\node (6) at (1.75,-2)  [draw,circle,inner sep=1.2pt,fill=black!100,label=0:{}] {};
		
		\path
		(1) edge [<-, black] node[above=0] {} (2)
        (1) edge [<-, black] node[above=0] {} (3)
        (2) edge [<-, black] node[above=0] {} (4)
        (3) edge [<-, black] node[above=0] {} (5)        
        (3) edge [<-, black] node[above=0] {} (6);
	\end{tikzpicture}
\end{center}

Now consider the 2-step nilpotent Lie algebra associated with the tree $T$. We define the closed 2-form
\begin{equation}
\sigma_{T,v_1}=\sum_{(v_i,v_j)\in E}\sigma_{i,j}.
\end{equation}
In the previous example, labeling the vertices as follows:

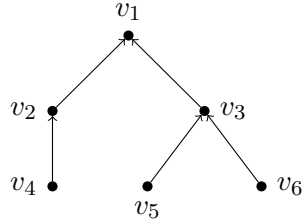
\begin{center}
	\begin{tikzpicture}
		\tikzset{every loop/.style={looseness=30}}
		\node (1) at (0,-0) [draw,circle,inner sep=1.2pt,fill=black!100,label=90:{$v_1$}] {};
		\node (2) at (-1,-1)  [draw,circle,inner sep=1.2pt,fill=black!100,label=180:{$v_2$}] {};
		\node (3) at (1,-1)  [draw,circle,inner sep=1.2pt,fill=black!100,label=0:{$v_3$}] {};
		\node (4) at (-1,-2)  [draw,circle,inner sep=1.2pt,fill=black!100,label=180:{$v_4$}] {};
		\node (5) at (0.25,-2)  [draw,circle,inner sep=1.2pt,fill=black!100,label=270:{$v_5$}] {};
		\node (6) at (1.75,-2)  [draw,circle,inner sep=1.2pt,fill=black!100,label=0:{$v_6$}] {};
		
		\path
		(1)  edge [<-, black] node[above=0] {} (2)
        (1)  edge [<-, black] node[above=0] {} (3)
        (2)  edge [<-, black] node[above=0] {} (4)
        (3)  edge [<-, black] node[above=0] {} (5)
        (3)  edge [<-, black] node[above=0] {} (6);
	\end{tikzpicture}
\end{center}
the resulting closed form is
\begin{eqnarray*}
	\sigma_{T,v_1}
	&=& \sigma_{4,2}+\sigma_{2,1}+\sigma_{5,3}+\sigma_{6,3}+\sigma_{3,1}\\
    &=& v_4^* \wedge [v_4,v_2]^* + v_2^* \wedge [v_2,v_1]^* + v_5^* \wedge [v_5,v_3]^* \\
    &&\quad +\, v_6^* \wedge [v_6,v_3]^* + v_3^* \wedge [v_3,v_1]^*.
\end{eqnarray*}
This may be represented as

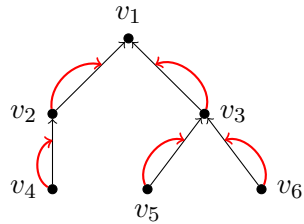
\begin{center}
	\begin{tikzpicture}
		\tikzset{every loop/.style={looseness=30}}
		\node (1) at (0,-0) [draw,circle,inner sep=1.2pt,fill=black!100,label=90:{$v_1$}] {};
		\node (2) at (-1,-1)  [draw,circle,inner sep=1.2pt,fill=black!100,label=180:{$v_2$}] {};
		\node (3) at (1,-1)  [draw,circle,inner sep=1.2pt,fill=black!100,label=0:{$v_3$}] {};
		\node (4) at (-1,-2)  [draw,circle,inner sep=1.2pt,fill=black!100,label=180:{$v_4$}] {};
		\node (5) at (0.25,-2)  [draw,circle,inner sep=1.2pt,fill=black!100,label=270:{$v_5$}] {};
		\node (6) at (1.75,-2)  [draw,circle,inner sep=1.2pt,fill=black!100,label=0:{$v_6$}] {};
		
		\path
		(1)  edge [<-, black] node[above=0] {} (2)
        (1)  edge [<-, black] node[above=0] {}(3)
        (2)  edge [<-, black] node[above=0] {} (4)
        (3) edge [<-, black] node[above=0] {} (5)
        (3)  edge [<-, black] node[above=0] {} (6);
  
		\coordinate (mid_v4_v2) at ($ (4)!0.65!(2) $);
		\coordinate (mid_v2_v1) at ($ (1)!0.35!(2) $);
		\coordinate (mid_v3_v1) at ($ (3)!0.65!(1) $);
		\coordinate (mid_v5_v3) at ($ (5)!0.65!(3) $);
		\coordinate (mid_v6_v3) at ($ (6)!0.65!(3) $);
		
		\draw[->,thick,bend left=65,red] (4) to (mid_v4_v2);
		\draw[->,thick,bend left=55,red] (2) to (mid_v2_v1);
		\draw[->,thick,bend right=65,red] (3) to (mid_v3_v1);
		\draw[->,thick,bend left=65,red] (5) to (mid_v5_v3);
		\draw[->,thick,bend right=65,red] (6) to (mid_v6_v3);
		
	\end{tikzpicture}
\end{center}

%-----------------------------------------------------------------
\subsection{Two trees}\label{subsec:tree}

Let $T_1$ and $T_2$ be two trees, and consider the $2$-step nilpotent Lie algebra associated with their disjoint union. 
Choosing roots $v_1$ and $v_2$, we define the symplectic form (see Remark \ref{rmk:sym1})
\begin{equation*}
\sigma=\sigma_{T_1,v_1}+\sigma_{T_2,v_2}+{v_1}^*\wedge {v_2}^*.
\end{equation*}

For example, consider the union of the following two trees, each directed from the leaves toward the root:

\medskip

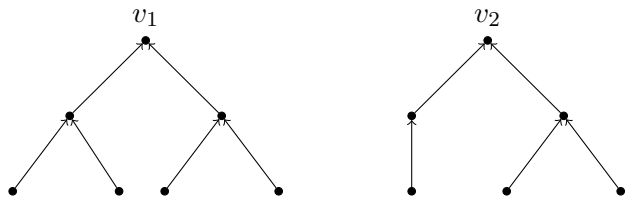
\begin{center}
\begin{tikzpicture}
\tikzset{every loop/.style={looseness=30}}
\node (1) at (0,0) [draw,circle,inner sep=1pt,fill=black!100,label=90:{$v_1$}] {};
\node (2) at (-1,-1) [draw,circle,inner sep=1pt,fill=black!100,label=90:{}] {};
\node (3) at (1,-1) [draw,circle,inner sep=1pt,fill=black!100,label=90:{}] {};
\node (4) at (-1.75,-2) [draw,circle,inner sep=1pt,fill=black!100,label=90:{}] {};
\node (5) at (-0.35,-2)  [draw,circle,inner sep=1pt,fill=black!100,label=180:{}] {};
\node (6) at (0.25,-2)  [draw,circle,inner sep=1pt,fill=black!100,label=180:{}] {};
\node (7) at (1.75,-2)  [draw,circle,inner sep=1pt,fill=black!100,label=0:{}] {};
\node (8) at (4.5,0)  [draw,circle,inner sep=1pt,fill=black!100,label=90:{$v_2$}] {};
\node (9) at (3.5,-1)  [draw,circle,inner sep=1pt,fill=black!100,label=180:{}] {};
\node (10) at (5.5,-1)  [draw,circle,inner sep=1pt,fill=black!100,label=0:{}] {};
\node (11) at (3.5,-2)  [draw,circle,inner sep=1pt,fill=black!100,label=180:{}] {};
\node (12) at (4.75,-2)  [draw,circle,inner sep=1pt,fill=black!100,label=0:{}] {};
\node (13) at (6.25,-2)  [draw,circle,inner sep=1pt,fill=black!100,label=0:{}] {};

\path
(1)  edge [<-, black] node[above=0] {} (2)
(1)  edge [<-, black] node[above=0] {} (3)
(2)  edge [<-, black] node[above=0] {} (4)
(2)  edge [<-, black] node[above=0] {} (5)
(3)  edge [<-, black] node[above=0] {} (6)
(3)  edge [<-, black] node[above=0] {} (7)
(8)  edge [<-, black] node[above=0] {} (9)
(8)  edge [<-, black] node[above=0] {} (10)
(9)  edge [<-, black] node[above=0] {} (11)
(10)  edge [<-, black] node[above=0] {} (12)
(10)  edge [<-, black] node[above=0] {} (13);
\end{tikzpicture}
\end{center}

The corresponding symplectic form may be represented as follows:

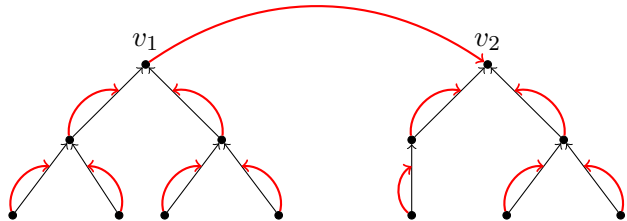
\begin{center}
\begin{tikzpicture}
\tikzset{every loop/.style={looseness=30}}
% Nodos
\node (1) at (0,0) [draw,circle,inner sep=1pt,fill=black!100,label=90:{$v_1$}] {};
\node (2) at (-1,-1) [draw,circle,inner sep=1pt,fill=black!100,label=90:{}] {};
\node (3) at (1,-1) [draw,circle,inner sep=1pt,fill=black!100,label=90:{}] {};
\node (4) at (-1.75,-2) [draw,circle,inner sep=1pt,fill=black!100,label=90:{}] {};
\node (5) at (-0.35,-2)  [draw,circle,inner sep=1pt,fill=black!100,label=180:{}] {};
\node (6) at (0.25,-2)  [draw,circle,inner sep=1pt,fill=black!100,label=180:{}] {};
\node (7) at (1.75,-2)  [draw,circle,inner sep=1pt,fill=black!100,label=0:{}] {};
\node (8) at (4.5,0)  [draw,circle,inner sep=1pt,fill=black!100,label=90:{$v_2$}] {};
\node (9) at (3.5,-1)  [draw,circle,inner sep=1pt,fill=black!100,label=180:{}] {};
\node (10) at (5.5,-1)  [draw,circle,inner sep=1pt,fill=black!100,label=0:{}] {};
\node (11) at (3.5,-2)  [draw,circle,inner sep=1pt,fill=black!100,label=180:{}] {};
\node (12) at (4.75,-2)  [draw,circle,inner sep=1pt,fill=black!100,label=0:{}] {};
\node (13) at (6.25,-2)  [draw,circle,inner sep=1pt,fill=black!100,label=0:{}] {};

\path
(1)  edge [<-, black] node[above=0] {} (2)
(1)  edge [<-, black] node[above=0] {} (3)
(2)  edge [<-, black] node[above=0] {} (4)
(2)  edge [<-, black] node[above=0] {} (5)
(3)  edge [<-, black] node[above=0] {} (6)
(3)  edge [<-, black] node[above=0] {} (7)
(8)  edge [<-, black] node[above=0] {} (9)
(8)  edge [<-, black] node[above=0] {} (10)
(9)  edge [<-, black] node[above=0] {} (11)
(10)  edge [<-, black] node[above=0] {} (12)
(10)  edge [<-, black] node[above=0] {} (13);

% --- Puntos medios de las aristas ---
\coordinate (mid_v2_v1) at ($ (1)!0.35!(2) $);
\coordinate (mid_v3_v1) at ($ (1)!0.35!(3) $);
\coordinate (mid_v7_v3) at ($ (3)!0.35!(7) $);
\coordinate (mid_v6_v3) at ($ (3)!0.35!(6) $);
\coordinate (mid_v4_v2) at ($ (2)!0.35!(4) $);
\coordinate (mid_v5_v2) at ($ (2)!0.35!(5) $);
\coordinate (mid_v8_v9) at ($ (8)!0.35!(9) $);
\coordinate (mid_v8_v10) at ($ (8)!0.35!(10) $);
\coordinate (mid_v9_v11) at ($ (9)!0.35!(11) $);
\coordinate (mid_v10_v12) at ($ (10)!0.35!(12) $);
\coordinate (mid_v10_v13) at ($ (10)!0.35!(13) $);

% --- Flechas curvas ---
\draw[->,thick,bend left=35,red] (1) to (8);
\draw[->,thick,bend left=55,red] (2) to (mid_v2_v1);
\draw[->,thick,bend right=55, red] (3) to (mid_v3_v1);
\draw[->,thick,bend left=55,red] (4) to (mid_v4_v2);
\draw[->,thick,bend right=55,red] (7) to (mid_v7_v3);
\draw[->,thick,bend left=55,red] (6) to (mid_v6_v3);
\draw[->,thick,bend right=55,red] (5) to (mid_v5_v2);
\draw[->,thick,bend left=55,red] (9) to (mid_v8_v9);
\draw[->,thick,bend right=55,red] (10) to (mid_v8_v10);
\draw[->,thick,bend left=55,red] (11) to (mid_v9_v11);
\draw[->,thick,bend left=55,red] (12) to (mid_v10_v12);
\draw[->,thick,bend right=55,red] (13) to (mid_v10_v13);
\end{tikzpicture}
\end{center}

%------------------------------------------------------
\subsection{One unicyclic graph}\label{subsec:cyc}

For the $2$-step nilpotent Lie algebra associated with the $n$-cycle ($n\ge 3$),

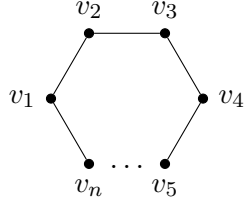
\begin{center}
\begin{tikzpicture}
\tikzset{every loop/.style={looseness=30}}
\node (1) at (1,0) [draw,circle,inner sep=1.2pt,fill=black!100,label=0:{$v_4$}] {};
\node (2) at (0.5,{sqrt(3)/2}) [draw,circle,inner sep=1.2pt,fill=black!100,label=90:{$v_3$}] {};
\node (3) at (-0.5,{sqrt(3)/2}) [draw,circle,inner sep=1.2pt,fill=black!100,label=90:{$v_2$}] {};
\node (4) at (-1,0) [draw,circle,inner sep=1.2pt,fill=black!100,label=180:{$v_1$}] {};
\node (5) at (-0.5,{-sqrt(3)/2}) [draw,circle,inner sep=1.2pt,fill=black!100,label=270:{$v_n$}] {};
\node (6) at (0.5,{-sqrt(3)/2}) [draw,circle,inner sep=1.2pt,fill=black!100,label=270:{$v_5$}] {};
\node (dots) at (0,-0.9) {$\ldots$};

\path
(1) edge (2)
(2) edge (3)
(3) edge (4)
(4) edge (5)
(6) edge (1);
\end{tikzpicture}
\end{center}

choose one of the two cyclic orientations and consider the symplectic form
\[
\sigma_{C_n} = \sum_{(v_i,v_j)\in E} \sigma_{i,j}.
\]
For the counterclockwise orientation, it may be represented as follows:

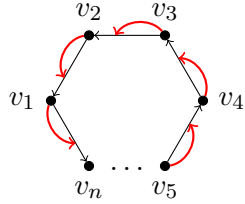
\begin{center}
\begin{tikzpicture}
\tikzset{every loop/.style={looseness=30}}
\node (1) at (-1,0) [draw,circle,inner sep=1.2pt,fill=black!100,label=180:{$v_1$}] {};
\node (2) at (-0.5,{sqrt(3)/2}) [draw,circle,inner sep=1.2pt,fill=black!100,label=90:{$v_2$}] {};
\node (3) at (0.5,{sqrt(3)/2}) [draw,circle,inner sep=1.2pt,fill=black!100,label=90:{$v_3$}] {};
\node (4) at (1,0) [draw,circle,inner sep=1.2pt,fill=black!100,label=0:{$v_4$}] {};
\node (5) at (0.5,{-sqrt(3)/2}) [draw,circle,inner sep=1.2pt,fill=black!100,label=270:{$v_5$}] {};
\node (6) at (-0.5,{-sqrt(3)/2}) [draw,circle,inner sep=1.2pt,fill=black!100,label=270:{$v_n$}] {};

\node (dots) at (0,-0.9) {$\ldots$};

\path
(1) edge[<-] (2)
(2) edge[<-] (3)
(3) edge[<-] (4)
(4) edge[<-] (5)
(6) edge[<-] (1);
		
\coordinate (mid_vn_v1) at ($ (1)!0.65!(6) $);
\coordinate (mid_v1_v2) at ($ (1)!0.35!(2) $);
\coordinate (mid_v2_v3) at ($ (2)!0.35!(3) $);
\coordinate (mid_v3_v4) at ($ (3)!0.35!(4) $);
\coordinate (mid_v4_v5) at ($ (4)!0.35!(5) $);
		
\draw[->,thick,bend right=55, red] (1) to (mid_vn_v1);
\draw[->,thick,bend right=55,red] (2) to (mid_v1_v2);
\draw[->,thick,bend right=55,red] (3) to (mid_v2_v3);
\draw[->,thick,bend right=55,red] (4) to (mid_v3_v4);
\draw[->,thick,bend right=55,red] (5) to (mid_v4_v5);
\end{tikzpicture}
\end{center}

Now consider the $2$-step nilpotent Lie algebra associated with a unicyclic graph obtained by gluing rooted trees $T_1,\dots,T_n$ at the vertices $v_1,\dots,v_n$ of the cycle. We orient the graph centripetally (see Section~\ref{subsec:dig}) and consider the symplectic form
\begin{equation*}
\sigma=\sigma_{C_n} + \sum_{i=1}^n \sigma_{T_i,v_i}.
\end{equation*}

For example, for the unicyclic graph

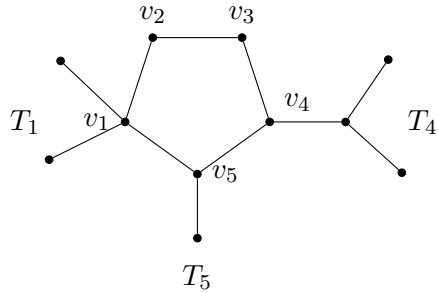
\begin{center}
 \begin{tikzpicture}
\tikzset{every loop/.style={looseness=30}}
\node (1) at (0,-1)[draw,circle,inner sep=1pt,fill=black!100,label=0:{$v_5$}] {};
\node (2) at (0.9511,-0.3090)[draw,circle,inner sep=1pt,fill=black!100,label=15:{$v_4$}] {};
\node (3) at (0.5878,0.8090)[draw,circle,inner sep=1pt,fill=black!100,label=90:{$v_3$}] {};
\node (4) at (-0.5878,0.8090)[draw,circle,inner sep=1pt,fill=black!100,label=90:{$v_2$}] {};
\node (5) at (-0.9511,-0.3090)[draw,circle,inner sep=1pt,fill=black!100,label=180:{$v_1$}] {};
\node (6) at (0,-1.85)  [draw,circle,inner sep=1pt,fill=black!100,label=270:{}] {};
\node (7) at (1.95,-0.3090)  [draw,circle,inner sep=1pt,fill=black!100,label=270:{}] {};
\node (8) at (2.511,0.517)  [draw,circle,inner sep=1pt,fill=black!100,label=270:{}] {};
\node (9) at (2.690,-0.982)  [draw,circle,inner sep=1pt,fill=black!100,label=270:{}] {};
\node (10) at (-1.80,0.50)  [draw,circle,inner sep=1pt,fill=black!100,label=270:{}] {};
\node (11) at (-1.9511,-0.8090)  [draw,circle,inner sep=1pt,fill=black!100,label=270:{}] {};
\node (12) at (2.5,-0.31)  [label=0:{$T_4$}] {};
\node (13) at (0,-2.8)  [label=90:{$T_5$}] {};
\node (14) at (-1.8,-0.31)  [label=180:{$T_1$}] {};

\path
(1) edge [-]  (2)
(2) edge [-]  (3)
(3) edge [-]  (4)
(4) edge [-]  (5)
(5) edge [-]  (1)
(1) edge [-]  (6)
(2) edge [-]  (7)
(7) edge [-]  (8)
(7) edge [-]  (9)
(5) edge [-]  (10)
(5) edge [-]  (11);
\end{tikzpicture}
\end{center}

we choose the orientation

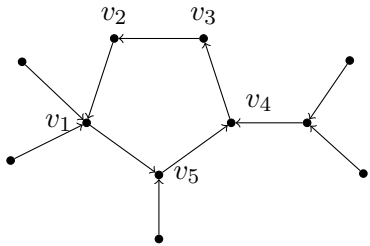
\begin{center}
 \begin{tikzpicture}
\tikzset{every loop/.style={looseness=30}}
\node (1) at (0,-1)[draw,circle,inner sep=1pt,fill=black!100,label=0:{$v_5$}] {};
\node (2) at (0.9511,-0.3090)[draw,circle,inner sep=1pt,fill=black!100,label=15:{$v_4$}] {};
\node (3) at (0.5878,0.8090)[draw,circle,inner sep=1pt,fill=black!100,label=90:{$v_3$}] {};
\node (4) at (-0.5878,0.8090)[draw,circle,inner sep=1pt,fill=black!100,label=90:{$v_2$}] {};
\node (5) at (-0.9511,-0.3090)[draw,circle,inner sep=1pt,fill=black!100,label=180:{$v_1$}] {};
\node (6) at (0,-1.85)  [draw,circle,inner sep=1pt,fill=black!100,label=270:{}] {};
\node (7) at (1.95,-0.3090)  [draw,circle,inner sep=1pt,fill=black!100,label=270:{}] {};
\node (8) at (2.511,0.517)  [draw,circle,inner sep=1pt,fill=black!100,label=270:{}] {};
\node (9) at (2.690,-0.982)  [draw,circle,inner sep=1pt,fill=black!100,label=270:{}] {};
\node (10) at (-1.80,0.50)  [draw,circle,inner sep=1pt,fill=black!100,label=270:{}] {};
\node (11) at (-1.9511,-0.8090)  [draw,circle,inner sep=1pt,fill=black!100,label=270:{}] {};

\path
(1) edge [->]  (2)
(2) edge [->]  (3)
(3) edge [->]  (4)
(4) edge [->]  (5)
(5) edge [->]  (1)
(1) edge [<-]  (6)
(2) edge [<-]  (7)
(7) edge [<-]  (8)
(7) edge [<-]  (9)
(5) edge [<-]  (10)
(5) edge [<-]  (11);
\end{tikzpicture}
\end{center}

The corresponding symplectic form is
\[
\sigma= \sigma_{C_5}+\sigma_{T_1,v_1}+\dots+\sigma_{T_5,v_5}.
\]
It may be represented as follows:

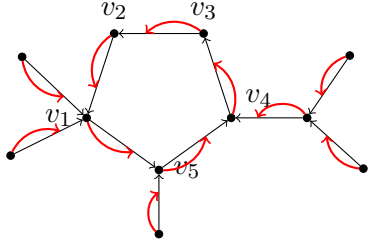
\begin{center}
 \begin{tikzpicture}
\tikzset{every loop/.style={looseness=30}}
\node (1) at (0,-1)[draw,circle,inner sep=1pt,fill=black!100,label=0:{$v_5$}] {};
\node (2) at (0.9511,-0.3090)[draw,circle,inner sep=1pt,fill=black!100,label=15:{$v_4$}] {};
\node (3) at (0.5878,0.8090)[draw,circle,inner sep=1pt,fill=black!100,label=90:{$v_3$}] {};
\node (4) at (-0.5878,0.8090)[draw,circle,inner sep=1pt,fill=black!100,label=90:{$v_2$}] {};
\node (5) at (-0.9511,-0.3090)[draw,circle,inner sep=1pt,fill=black!100,label=180:{$v_1$}] {};
\node (6) at (0,-1.85)  [draw,circle,inner sep=1pt,fill=black!100,label=270:{}] {};
\node (7) at (1.95,-0.3090)  [draw,circle,inner sep=1pt,fill=black!100,label=270:{}] {};
\node (8) at (2.511,0.517)  [draw,circle,inner sep=1pt,fill=black!100,label=270:{}] {};
\node (9) at (2.690,-0.982)  [draw,circle,inner sep=1pt,fill=black!100,label=270:{}] {};
\node (10) at (-1.80,0.50)  [draw,circle,inner sep=1pt,fill=black!100,label=270:{}] {};
\node (11) at (-1.9511,-0.8090)  [draw,circle,inner sep=1pt,fill=black!100,label=270:{}] {};

\path
(1) edge [->]  (2)
(2) edge [->]  (3)
(3) edge [->]  (4)
(4) edge [->]  (5)
(5) edge [->]  (1)
(1) edge [<-]  (6)
(2) edge [<-]  (7)
(7) edge [<-]  (8)
(7) edge [<-]  (9)
(5) edge [<-]  (10)
(5) edge [<-]  (11);

\coordinate (mid_v1_v2) at ($ (1)!0.65!(2) $);
\coordinate (mid_v2_v3) at ($ (2)!0.65!(3) $); 
\coordinate (mid_v3_v4) at ($ (3)!0.65!(4) $);
\coordinate (mid_v4_v5) at ($ (4)!0.65!(5) $); 
\coordinate (mid_v5_v1) at ($ (5)!0.65!(1) $);
\coordinate (mid_v2_v7) at ($ (2)!0.35!(7) $); 
\coordinate (mid_v7_v8) at ($ (7)!0.35!(8) $); 
\coordinate (mid_v7_v9) at ($ (7)!0.35!(9) $); 
\coordinate (mid_v1_v6) at ($ (1)!0.35!(6) $); 
\coordinate (mid_v5_v10) at ($ (5)!0.35!(10) $); 
\coordinate (mid_v5_v11) at ($ (5)!0.35!(11) $); 

\draw[->,thick,bend right=40,red] (1) to (mid_v1_v2);
\draw[->,thick,bend right=40,red] (2) to (mid_v2_v3);
\draw[->,thick,bend right=40,red] (3) to (mid_v3_v4);
\draw[->,thick,bend right=40,red] (4) to (mid_v4_v5);
\draw[->,thick,bend right=40,red] (5) to (mid_v5_v1);
\draw[->,thick,bend left=50,red] (6) to (mid_v1_v6);
\draw[->,thick,bend right=60,red] (7) to (mid_v2_v7);
\draw[->,thick,bend right=40,red] (8) to (mid_v7_v8);
\draw[->,thick,bend left=40,red] (9) to (mid_v7_v9);
\draw[->,thick,bend right=40,red] (10) to (mid_v5_v10);
\draw[->,thick,bend left=40,red] (11) to (mid_v5_v11);
\end{tikzpicture}
\end{center}
%------------------------------------------------------

%===========================================================
\section{The family of $k$-step nilpotent Lie algebras $\g(\kappa,G)$}

Given a digraph $G$ and a function $\kappa:E\to\N_{\ge 2}$, we define a nilpotent Lie algebra $\g(\kappa,G)$. 
This algebra is a quotient of the $k$-step nilpotent Lie algebra associated with the graph $G$, denoted by $\g(k,G)$, 
where $k$ is the maximum value of $\kappa$.
%---------------------------------------------------------------------------
\subsection{The $k$-step nilpotent Lie algebra associated with a graph}

Let $G=(V,E)$ be a graph, with vertex set $V=\{v_1,\dots,v_m\}$. For $k\ge 2$, the $k$-step nilpotent Lie algebra associated with $G$ is the quotient
\[
\g(k,G)=L_{(k)}(m)/I,
\]
where $L_{(k)}(m)$ is the free $k$-step nilpotent Lie algebra on $m$ generators, and $I$ is the ideal generated by
\[
\{[v_i,v_j]: \{v_i,v_j\}\notin E\}.
\]

The free $k$-step nilpotent Lie algebra is naturally graded:
\[
L_{(k)}(m)=L_1\oplus L_2\oplus\cdots\oplus L_k,
\]
and $I$ is a homogeneous ideal:
\[
I=I_2\oplus\cdots\oplus I_k,
\]
where $I_j=I\cap L_j\subseteq L_j$. Hence $\g(k,G)$ is naturally graded:
\[
\begin{aligned}
\g(k,G)&=L_{(k)}(m)/I\\
&=(L_1\oplus\cdots\oplus L_k)/(I_2\oplus\cdots\oplus I_k)\\
&\cong L_1\oplus L_2/I_2\oplus\cdots\oplus L_k/I_k\\
&=V_1\oplus V_2\oplus\cdots\oplus V_k.
\end{aligned}
\]
Moreover, this decomposition is a Carnot grading, that is,
\[
\g(k,G)^i = V_i\oplus V_{i+1}\oplus\cdots\oplus V_k,
\qquad \text{for all } 1\le i\le k.
\]

\begin{remark}
If the graph $G$ has connected components $G_1,\dots,G_\ell$, then
\[
\g(k,G)\simeq \g(k,G_1)\oplus\cdots\oplus\g(k,G_\ell).
\]
In particular, each isolated vertex of $G$ corresponds to a $1$-dimensional abelian direct summand of $\g(k,G)$.
\end{remark}

%-----------------------------------------
\subsubsection{A basis for $\g(k,G)$}\label{subsec:basis}

Given a graph $G$, let $G'$ denote its complementary graph, and let $m$ be the number of vertices of $G$. 
In \cite[Theorem 5.12, Section 5]{W}, Wade presented a combinatorial algorithm that, 
given the ordered vertex set $v_1<v_2<\dots<v_m$ as input, outputs a basis for $\g(k,G)$. 

Starting from $G'$, the algorithm recursively constructs Lyndon elements (represented by words in the vertices) and then the corresponding Lyndon brackets, thereby yielding a basis of $\g(k,G)$. 
More precisely, it provides bases $B_j$ for each $V_j$, where the corresponding Lyndon elements all have length $j$.
The basis for $\g(k,G)$, is then
\begin{equation}\label{eqn:wade-basis}
    B=B_1 \cup\dots\cup B_k.
\end{equation}

\begin{example}\label{ex:grp}
	As an example (see \cite[Examples 5.13 and 5.16]{W}), consider the simple graph $G$ (left) and its complementary graph $G'$ (right).
	
	\vskip.4cm
	
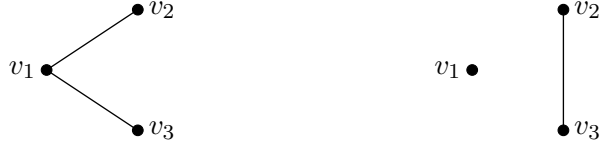
	\begin{center}
		\begin{tikzpicture}[scale=.8]
			\filldraw (3.5,-2) circle (2.5pt);
			\node at (3.1, -2) {$v_1$};
			\filldraw (5,-1) circle (2.5pt);
			\node at (5.4,-1) {$v_2$};
			\filldraw (5,-3) circle (2.5pt);
			\node at (5.4,-3) {$v_3$};
			\draw [line width=0.5pt] (3.5,-2)--(5,-3);
			\draw [line width=0.5pt] (3.5,-2)--(5,-1);

			\filldraw (10.5,-2) circle (2.5pt);
			\node at (10.1, -2) {$v_1$};
			\filldraw (12,-1) circle (2.5pt);
			\node at (12.4,-1) {$v_2$};
			\filldraw (12,-3) circle (2.5pt);
			\node at (12.4,-3) {$v_3$};
			\draw [line width=0.5pt] (12,-1)--(12,-3);
		\end{tikzpicture}
	\end{center}
	The corresponding basis for $\g(4,G)$, $B=B_1\cup\dots\cup B_4$, is given by
	\begin{align*}
		B_1=&\{v_1,v_2,v_3\};\\
		B_2=&\big\{[v_1,v_2],[v_1,v_3]\big\};\\
		B_3=&\big\{[v_1,[v_1,v_2]], [v_1,[v_1,v_3]], [[v_1,v_2],v_2], [[v_1,v_3],v_2], [[v_1,v_3],v_3]\big\};\\
		B_4=&\big\{[v_1, [v_1, [v_1, v_2]]], [v_1, [v_1, [v_1, v_3]]], [[[v_1, v_2], v_2], v_2], \\
        &[[[v_1, v_3], v_2], v_2],
		[[[v_1, v_3], v_3], v_2], 
		[[[v_1, v_3], v_3], v_3], [[v_1, [v_1, v_2]], v_2], \\
        &[[v_1, [v_1, v_3]], v_2],
		[[v_1, [v_1, v_3]], v_3], [[v_1, v_2], [v_1, v_3]]\big\}.
	\end{align*}
\end{example}

\begin{remark}\label{rmk:basis1}
If $i<j$, and hence $v_i<v_j$, the element
\[
[\dots [[v_i,v_j],v_j],\dots,v_j]
\]
is a Wade basis element.

If \(i>j\), and hence $v_i>v_j$, the element
\[
[\dots [[v_i,v_j],v_j],\dots,v_j]
\]
is not a Wade basis element, but if the multiplicity of $v_j$ is $m_j$, then
\[ 
[\dots [[v_i,v_j],v_j],\dots,v_j] = (-1)^{m_j} [v_j,[v_j,\dots[v_j,v_i]\dots]],
\]
and \([v_j,[v_j,\dots[v_j,v_i]\dots]]\) is a Wade basis element.

\end{remark}

\medskip

Each basis element $\omega$ has a multidegree $\mult(\omega)$, the $m$-tuple of multiplicities of each vertex in the underlying word. 
In Example~\ref{ex:grp}, we have
\[
\mult([[v_1,v_2],v_2])=(1,2,0)
\qquad\text{and}\qquad
\mult([[v_1, [v_1, v_3]], v_2])=(2,1,1).
\]
Given two basis elements $\omega_i$ and $\omega_j$, their bracket is
\[
[\omega_i,\omega_j]=\sum_k c_k \omega_k,
\]
where $c_k\ne 0$ only if
\begin{equation}\label{eqn:multidegree}
	\mult(\omega_k)= \mult(\omega_i)+\mult(\omega_j).
\end{equation}

We further denote by $\mult_i$ the $i$-th coordinate of $\mult$, so that $\mult_i(\omega)$ is the multiplicity of the vertex $v_i$ in $\omega$. For $\omega=[[v_1,v_2],v_2]$, we have $\mult_1(\omega)=1$, $\mult_2(\omega)=2$, and $\mult_3(\omega)=0$.

%----------------------------------------------------------
\subsection{The family of Lie algebras $\g(\kappa,G)$}\label{subsec:g-kappa}

Given a digraph $G$ and a natural number $k\ge 2$, consider the $k$-step nilpotent Lie algebra $\g(k,G)$ associated with the underlying graph $G$
with vertices $v_1<\dots<v_m$.
Let $B = B_1 \cup \cdots \cup B_k$ be the corresponding Wade basis of $\g(k,G)$.

For each edge $(v_i,v_j)\in E$, we define inductively
\[
w_{i,j}^1=v_i, \qquad w_{i,j}^2=[v_i,v_j], \qquad w_{i,j}^{l}=[ w_{i,j}^{l-1},v_j], \quad \text{for } 3\leq l\leq k.
\]
Note that, if $i<j$, then $w_{i,j}^l$ is a Wade basis element for all $1\leq l\leq k$. 
Moreover, $w_{i,j}^l\in B_l$ for all $2\leq l\leq k$.
If $i>j$, then 
\[ w_{i,j}^l=(-1)^{l-1} [v_j,[v_j,\dots[v_j,v_i]\dots]], \] 
where $[v_j,[v_j,\dots[v_j,v_i]\dots]]$ is a Wade basis element in $B_l$ (see Remark \ref{rmk:basis1}).

\begin{definition}\label{def:wade-adapted}
  Given a digraph $G=(V,E)$, with vertices $v_1<\dots <v_m$ and corresponding Wade basis $B$,
  we define the adapted Wade basis to be the basis obtained from $B$ by replacing for each 
  oriented edge $(v_i,v_j)\in E$ with $i>j$ each element $[v_j,[v_j,\dots[v_j,v_i]\dots]]$ by $w_{i,j}^l$, for all $2\leq l\leq k$.
\end{definition}

Now, for every oriented edge $(v_i,v_j)\in E$, $w_{i,j}^l$ is an element of the adapted Wade basis for all $l$. 

\medskip

The multidegree $\mult$ is well defined on the adapted Wade basis and satisfies \eqref{eqn:multidegree}.

\begin{remark}\label{rmk:1}
For any $(v_i,v_j)\in E$ and $l\geq 2$, we have
\begin{align*}
\mult_j(w_{i,j}^l)&=l-1,\\
\mult_i(w_{i,j}^l)&=1,\\
\mult_h(w_{i,j}^l)&=0, \quad \text{for all } h\neq i,j.
\end{align*}
Moreover, $w_{i,j}^l$ is the unique basis element with this multidegree.
\end{remark}

From now on, by basis elements we will mean elements of the adapted Wade basis.

\begin{lemma}\label{lemma:1}
Let $G$ be a digraph. Given a natural number $2\leq l\leq k$ and two basis elements $b_s$ and $b_t$ such that
\[
\mult(w_{i,j}^l)=\mult(b_s)+\mult(b_t),
\]
then $b_s=v_j$ or $b_s=w_{i,j}^{l-1}$.
\end{lemma}

\begin{proof}
Let $b_s,b_t\in B$ be such that
\[
\mult(w_{i,j}^l)=\mult(b_s)+\mult(b_t).
\]
By Remark~\ref{rmk:1}, we have
\begin{align*}
l-1&=\mult_j(b_s)+\mult_j(b_t),\\
1&=\mult_i(b_s)+\mult_i(b_t),\\
0&=\mult_h(b_s)+\mult_h(b_t), \quad \text{for all } h\neq i,j.
\end{align*}
Hence $\mult_h(b_s)=\mult_h(b_t)=0$ for all $h\neq i,j$. Also, either $\mult_i(b_s)=0$ or $\mult_i(b_t)=0$. In the first case, the word associated with $b_s$ contains only the letter $v_j$, and therefore $b_s=v_j$. In the second case, similarly, $b_t=v_j$, and then
\begin{align*}
l-2&=\mult_j(b_s),\\
1&=\mult_i(b_s),\\
0&=\mult_h(b_s), \quad \text{for all } h\neq i,j.
\end{align*}
Therefore, by Remark~\ref{rmk:1}, we obtain $b_s=w_{i,j}^{l-1}$.
\end{proof}

\begin{proposition}
Let $G=(V,E)$ be a digraph and let $k\ge 2$ be a natural number. 
For each oriented edge $(v_i,v_j)\in E$ and each natural number $1\leq \kappa_{i,j}\leq k$, the subspace
\[
I_{i,j}^{\kappa_{i,j}}=
\left\langle B\setminus\big(\{v_j\}\cup\{w_{i,j}^l: 1\leq l\leq \kappa_{i,j}\}\big)\right\rangle
\]
is an ideal of $\g(k,G)$.
\end{proposition}

\begin{proof}
Given
\[
b_s\in B\setminus\big(\{v_j\}\cup\{w_{i,j}^l: 1\leq l\leq \kappa_{i,j}\}\big)
\qquad\text{and}\qquad
b_t\in B,
\]
their bracket is of the form
\[
[b_s,b_t]=\sum_{b_r\in B} c_r b_r,
\]
where $c_r\ne 0$ only if
\[
\mult(b_r)= \mult(b_s)+\mult(b_t).
\]
Therefore, it is enough to prove that this equality never holds for $b_r=v_j$ or for $b_r=w_{i,j}^l$, with $1\leq l\leq \kappa_{i,j}$.

\begin{enumerate}
\item Suppose that
\[
\mult(v_j)=\mult(b_s)+\mult(b_t).
\]
Then
\begin{align*}
1&=\mult_j(b_s)+\mult_j(b_t),\\
0&=\mult_h(b_s)+\mult_h(b_t), \quad \text{for all } h\neq j.
\end{align*}
This implies that $b_s=0$ or $b_t=0$, a contradiction.

\item Suppose that
\[
\mult(w_{i,j}^1)=\mult(b_s)+\mult(b_t).
\]
Then
\begin{align*}
1&=\mult_i(b_s)+\mult_i(b_t),\\
0&=\mult_h(b_s)+\mult_h(b_t), \quad \text{for all } h\neq i.
\end{align*}
Again, this implies that $b_s=0$ or $b_t=0$, a contradiction.

\item Suppose that
\[
\mult(w_{i,j}^l)=\mult(b_s)+\mult(b_t)
\]
for some $2\leq l\leq \kappa_{i,j}$. By Lemma~\ref{lemma:1}, we have $b_s=v_j$ or $b_s=w_{i,j}^{l-1}$. In either case,
\[
b_s\in \{v_j\}\cup\{w_{i,j}^r: 1\leq r\leq \kappa_{i,j}\},
\]
a contradiction.
\end{enumerate}
\end{proof}

\begin{definition}
A \emph{chain digraph} is a pair \((\kappa,G)\), where \(G=(V,E)\) is a digraph and $\kappa:E\rightarrow \N_{\geq 2}$ is a function. 
We denote its value on the edge $(v_i,v_j)\in E$ by $\kappa_{i,j}$ and its maximum value by $k$. 
\end{definition}

\begin{definition}
Given a chain digraph $(\kappa,G)$, 
its associated Lie algebra $(\kappa,G)$ is the quotient Lie algebra
\[
\g(\kappa,G)=\g(k,G)/ I_\kappa,
\]
where $I_\kappa$ is the ideal of $\g(k,G)$ given by
\[
I_\kappa=\left(\bigcap_{\substack{(v_i,v_j)\in E}} I_{i,j}^{\kappa_{i,j}}\right) \cap \left(V_3\oplus\dots\oplus V_k\right).
\]
\end{definition}

\begin{proposition}\label{prop:bracket}
Given a chain digraph $(\kappa,G)$, the Lie algebra $\g(\kappa,G)$ is $k$-step nilpotent. 
A basis for $\g(\kappa,G)$ is
\[
B_1\cup \bigcup_{l=2}^{k} B_l^\kappa,
\]
where $B_1=\{v_1,\dots, v_m\}$ and
\[
B_l^\kappa= \big\{w_{i,j}^l: (v_i,v_j)\in E \text{ and } l\leq \kappa_{i,j} \big\}.
\]

The Lie bracket of $\g(\kappa,G)$ is determined by the nonzero brackets
\[
[w_{i,j}^l, v_j]=w_{i,j}^{l+1}, \qquad \text{if } (v_i,v_j)\in E \text{ and } 1\leq l\leq \kappa_{i,j}-1.
\]
Moreover, $\g(\kappa,G)$ is a naturally graded Lie algebra, and a Carnot grading is
\[
\g(\kappa,G)=\langle B_1\rangle \oplus \langle B_2^{\kappa}\rangle \oplus\dots\oplus\langle B_k^{\kappa}\rangle .
\]
\end{proposition}

From now on, given a chain digraph \((\kappa,G)\), the basis given in the previous proposition will be called \emph{standard basis of \(\g(\kappa,G)\)}. The \emph{basis elements corresponding to the edge \((v_i,v_j)\)} are \(v_i, v_j\) and \(w_{i,j}^l\), for \(2\leq l\leq\kappa_{i,j}\). 

\begin{corollary}\label{coro:diff}
It holds that
\[
d(v_i^*)=0 \quad\text{and}\quad d({w_{i,j}^l}^*)={w_{i,j}^{l-1}}^* \wedge v_j^*, \text{ for all $2\le l\le \kappa_{i,j}$}.
\]
\end{corollary}

\begin{remark}
If $\kappa$ is constantly equal to 2, then $\g(\kappa,G)$ is $2$-step nilpotent and isomorphic to the graph Lie algebra $\g(2,G)$ associated 
with the underlying simple graph $G$.

The vector space
\[
\Omega=\bigoplus_{l=3}^k \left\langle\{w_{i,j}^l: (v_i,v_j)\in E \text{ and } l\leq \kappa_{i,j} \}\right\rangle
\]
is an abelian ideal of $\g(\kappa,G)$, and the quotient $\g(\kappa,G)/\Omega$ is isomorphic to $\g(2,G)$.
\end{remark}

The Lie algebra $\g(\kappa,G)$ can be represented 
by adding to the digraph $G$ a chain of vertices of length $\kappa_{i,j}-2$ to each edge $(v_i,v_j)$.

\begin{example}
Consider the digraph $G$ presented in Section~\ref{sec:2-gr} and the chain function $\kappa$ that takes the value $2$ on every edge except $(v_4,v_2)$ and $(v_3,v_1)$, where
\[
\kappa((v_4,v_2))=6
\qquad\text{and}\qquad
\kappa((v_3,v_1))=4.
\]
The corresponding 17-dimensional $6$-step nilpotent Lie algebra $\g(\kappa,G)$ is represented by
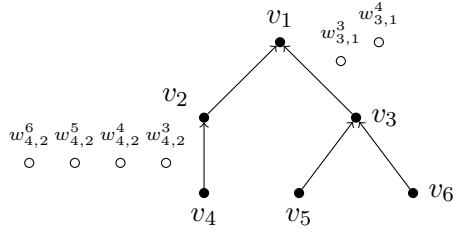
\begin{center}
\begin{tikzpicture}
\tikzset{every loop/.style={looseness=30}}
\node (1) at (0,0) [draw,circle,inner sep=1.2pt,fill=black!100,label=90:{$v_1$}] {};
\node (2) at (-1,-1)  [draw,circle,inner sep=1.2pt,fill=black!100,label=160:{$v_2$}] {};
\node (3) at (1,-1)  [draw,circle,inner sep=1.2pt,fill=black!100,label=0:{$v_3$}] {};
\node (4) at (-1,-2)  [draw,circle,inner sep=1.2pt,fill=black!100,label=270:{$v_4$}] {};
\node (5) at (0.25,-2)  [draw,circle,inner sep=1.2pt,fill=black!100,label=270:{$v_5$}] {};
\node (6) at (1.75,-2)  [draw,circle,inner sep=1.2pt,fill=black!100,label=0:{$v_6$}] {};
\node (7) at (-1.5,-1.6) [draw,circle,inner sep=1.2pt, label=90:{\tiny{$w_{4,2}^3$}}] {};
\node (8) at (-2.1,-1.6) [draw,circle,inner sep=1.2pt, label=90:{\tiny{$w_{4,2}^4$}}] {};
\node (9) at (-2.7,-1.6) [draw,circle,inner sep=1.2pt, label=90:{\tiny{$w_{4,2}^5$}}] {};
\node (10) at (-3.3,-1.6) [draw,circle,inner sep=1.2pt, label=90:{\tiny{$w_{4,2}^6$}}] {};
\node (11) at (0.8,-0.25) [draw,circle,inner sep=1.2pt, label=90:{\tiny{$w_{3,1}^3$}}] {};
\node (12) at (1.3,0) [draw,circle,inner sep=1.2pt, label=90:{\tiny{$w_{3,1}^4$}}] {};
		
\path
(1)  edge [<-, black] node[above=0] {} (2)
(1)  edge [<-, black] node[above=0] {} (3)
(2)  edge [<-, black] node[above=0] {} (4)
(3)  edge [<-, black] node[above=0] {} (5)
(3)  edge [<-, black] node[above=0] {} (6);
\end{tikzpicture}
\end{center}

The only non-zero basis brackets involving $w_{4,2}^i$ and $w_{3,1}^i$ are:
\[ [w_{4,2}^1,v_2]=w_{4,2}^2, \quad [w_{4,2}^2,v_2]=w_{4,2}^3, \quad \dots, \quad [w_{4,2}^5,v_2]=w_{4,2}^6; \]
\[ [w_{3,1}^1,v_1]=w_{3,1}^2, \quad [w_{3,1}^2,v_1]=w_{3,1}^3, \quad [w_{3,1}^3,v_1]=w_{3,1}^4. \]
Recall that $w_{4,2}^1=v_4$, $w_{4,2}^2=[v_4,v_2]$ and that $w_{3,1}^1=v_3$, $w_{3,1}^2=[v_3,v_1]$.

The other non-zero brackets are the usual ones (corresponding to the remaining edges)
\[[w_{2,1}^1,v_1]=w_{2,1}^2, \quad [w_{5,3}^1,v_3]=w_{5,3}^2, \quad [w_{6,3}^1,v_3]=w_{6,3}^2. \]
\end{example}

%==================================================
\section{Families of symplectic $k$-step nilpotent Lie algebras }\label{sec:fam-sym}

In this section we identify conditions on a digraph $G$ with chain function $\kappa$, such that $\g(\kappa,G)$ is symplectic. 
Moreover, for each \emph{chain digraph} $(\kappa,G)$ in this family we exhibit an explicit symplectic form of $\g(\kappa,G)$.

\begin{theorem}
Let \((\kappa,G)\) be a chain digraph. If \(\g(\kappa,G)\) is symplectic, then \(\dim(\g(\kappa,G))\) is even and \(G\) has no more edges than vertices.
\end{theorem}

\begin{proof}
It is well known that any Lie algebra \(\g\) with a nondegenerated closed \(2\)-form must be even dimensional and must satisfy \(\dim(Z(\g))\leq \dim(\g/ \g')\). 
In our case, the set \(S=\{w_{i,j}^{\kappa_{i,j}}: (v_i,v_j)\in E\}\) is contained in \(Z(\g)\) and \(\dim(\langle S\rangle)=|E|\). Also, \(\dim(\g/ \g')=|V|\).

Therefore, \(|E|\leq |V|\) and the dimension of \(\g(\kappa,G)\) is even.
\end{proof}

\begin{remark}\label{rmk:graphs}
Note that if a graph $G$ has no more edges than vertices, each connected component has at most one cycle. 
\end{remark}

\begin{theorem}\label{thm:even}
	Let \((\kappa,G)\) be a chain digraph and let $(v_i,v_j)$ be an oriented edge.
    Then, if $\kappa_{i,j}$ is even, the \(2\)-form 
    \begin{eqnarray}
    	 \sigma_{i,j} &=& \sum_{l=1}^{\frac{\kappa_{i,j}}2} (-1)^{l+1} \,{w_{i,j}^{l}}^* \wedge {w_{i,j}^{(\kappa_{i,j}+1-l)}}^* 
                                                                                                                   \label{eqn:sigma-even}\\ 
    	                &=&  v_i^* \wedge {w_{i,j}^{\kappa_{i,j}}}^* + 
    	                 \sum_{l=2}^{\frac{\kappa_{i,j}}2} (-1)^{l+1}\,{w_{i,j}^{l}}^* \wedge {w_{i,j}^{(\kappa_{i,j}+1-l)}}^* \nonumber
    \end{eqnarray}                    
    is closed on the Lie algebra \(\g(\kappa,G)\).
\end{theorem}

\begin{proof} 
First, we have 
\begin{eqnarray*}
    	 d(\sigma_{i,j}) &=& d\left(\sum_{l=1}^{\frac{\kappa_{i,j}}2} (-1)^{l+1} {w_{i,j}^{l}}^* \wedge {w_{i,j}^{(\kappa_{i,j}+1-l)}}^*\right) \\ 
    	                &=& \underbrace{ d\left(v_i^* \wedge {w_{i,j}^{\kappa_{i,j}}}^*\right)}_A + 
    	                 \underbrace{\sum_{l=2}^{\frac{\kappa_{i,j}}2} (-1)^{l+1} \, d\left({w_{i,j}^{l}}^* \wedge {w_{i,j}^{(\kappa_{i,j}+1-l)}}^*\right)}_{B}
\end{eqnarray*}
On the one hand, from Corollary, \ref{coro:diff} it follows that
\[ 
A = -v_i^* \wedge {w_{i,j}^{(\kappa_{i,j}-1)}}^*\wedge  v_j^* = v_i^* \wedge v_j^* \wedge {w_{i,j}^{(\kappa_{i,j}-1)}}^*.
\]
On the other hand, also from Corollary \ref{coro:diff}, it follows that
\begin{eqnarray*}
                         B&=&\sum_{l=2}^{\frac{\kappa_{i,j}}2} (-1)^{l+1} \, \left( d\left({w_{i,j}^{l}}^*\right) \wedge {w_{i,j}^{(\kappa_{i,j}+1-l)}}^*- {w_{i,j}^{l}}^* \wedge d\left({w_{i,j}^{(\kappa_{i,j}+1-l)}}^*\right)\right)\\
                         &=& \sum_{l=2}^{\frac{\kappa_{i,j}}2} (-1)^{l+1} \left(  {w_{i,j}^{l-1}}^* \wedge  v_j^*\wedge {w_{i,j}^{(\kappa_{i,j}+1-l)}}^*- {w_{i,j}^{l}}^* \wedge {w_{i,j}^{(\kappa_{i,j}+1-l-1)}}^*\wedge  v_j^*\right)\\
                            &=& \sum_{l=2}^{\frac{\kappa_{i,j}}2} (-1)^{l+1}\,  {w_{i,j}^{l-1}}^* \wedge  v_j^*\wedge {w_{i,j}^{(\kappa_{i,j}+1-l)}}^*+ \sum_{l=2}^{\frac{\kappa_{i,j}}2} (-1)^{l+1}\, {w_{i,j}^{l}}^* \wedge  v_j^* \wedge  {w_{i,j}^{(\kappa_{i,j}-l)}}^* \\
                         &=& - v_i^* \wedge v_j^* \wedge {w_{i,j}^{(\kappa_{i,j}-1)}}^*.
\end{eqnarray*}  
Therefore, the \(2\)-form \(\sigma_{i,j}\) is closed on the Lie algebra \(\g(\kappa,G)\).
    \end{proof}

Observe that if \(\kappa_{i,j}=2\), then \(\sigma_{i,j}\) coincides with the \(2\)-form \(\sigma_{i,j}\) defined for the \(2\)-step graph Lie algebras in (\ref{eqn:sigma-ij}).  
   
\begin{theorem}\label{thm:odd}
	Let \((\kappa,G)\) be a chain digraph and $(v_i,v_j)$ be an oriented edge.
    Then, if $\kappa_{i,j}$ is odd, $\kappa'_{i,j}=\kappa_{i,j}-1$ is even and the \(2\)-form
    \begin{equation}\label{eqn:sigma-odd}
    	\sigma_{i,j} = v_j^* \wedge {w_{i,j}^{\kappa_{i,j}}}^* + \sum_{l=1}^{\frac{\kappa'_{i,j}}2} (-1)^{l+1}\,{w_{i,j}^{l}}^* \wedge {w_{i,j}^{(\kappa'_{i,j}+1-l)}}^*
    \end{equation}                  
    is closed on the Lie algebra \(\g(\kappa,G)\).
\end{theorem}

\begin{proof}
We have that
\[
d(\sigma_{i,j}) = d( v_j^* \wedge {w_{i,j}^{\kappa_{i,j}}}^*) + d \left(\sum_{l=1}^{\frac{\kappa'_{i,j}}2}  (-1)^{l+1}\, {w_{i,j}^{l}}^* \wedge {w_{i,j}^{(\kappa'_{i,j}+1-l)}}^*\right).
\]
By the previous theorem, the second summand is zero. 
As for the first one, and by Corollary \ref{coro:diff}, we have that
\[
d( v_j^* \wedge {w_{i,j}^{\kappa_{i,j}}}^*) = -  \,v_j^* \wedge {w_{i,j}^{\kappa_{i,j}-1}}^*\wedge v_j^* = 0.
\]
Therefore, the \(2\)-form \(\sigma_{i,j}\) is closed on the Lie algebra \(\g(\kappa,G)\).
\end{proof}

These forms, associated to the oriented edge \((v_i,v_j)\), may be depicted as follows. 
For the left one, $\kappa_{i,j}$ is even and $\ge 8$, while for the right one, $\kappa_{i,j}$ is odd and $\ge 9$.

\medskip

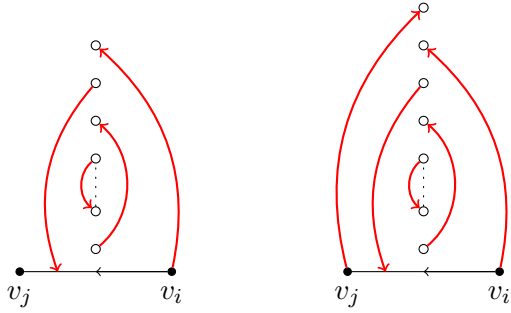
\begin{center}
	\begin{tikzpicture}
		\tikzset{every loop/.style={looseness=-10}}
		\node (1) at (0,0) [draw,circle,inner sep=1pt,fill=black,label=-90:{$v_j$}] {};
		\node (2) at (2,0) [draw,circle,inner sep=1pt,fill=black,label=-90:{$v_i$}] {};
		\node (3) at (1,0.3) [draw,circle,inner sep=1.2pt,label=0:{}] {};
		\node (4) at (1,0.8) [draw,circle,inner sep=1.2pt,label=0:{}] {};
		\node (5) at (1,1.5) [draw,circle,inner sep=1.2pt,label=0:{}] {};
		\node (6) at (1,2) [draw,circle,inner sep=1.2pt,label=180:{}] {};
        \node (7) at (1,2.5) [draw,circle,inner sep=1.2pt,label=0:{}] {};
		\node (8) at (1,3) [draw,circle,inner sep=1.2pt,label=180:{}] {};
		
		\path
		(1) edge [-, black]  (2)
        (2) edge [->, black] (1,0)
		(4) edge [dash pattern=on 0.5pt off 3pt, line cap=round] (5);
		
		\coordinate (mid_v1_v2) at ($ (1)!0.25!(2) $);  % 3/4 arista v1-v2
		
		%Flechas
		\draw[->,thick,bend right=30, red] (2) to (8);    
		\draw[->,thick,bend right=30,red] (7) to (mid_v1_v2);  
		\draw[->,thick,bend right=50,red] (3) to (6);
        \draw[->,thick,bend right=50,red] (5) to (4);
	\end{tikzpicture}  
	\qquad\qquad
   \begin{tikzpicture}
		\tikzset{every loop/.style={looseness=-10}}
		\node (1) at (0,0) [draw,circle,inner sep=1pt,fill=black,label=-90:{$v_j$}] {};
		\node (2) at (2,0) [draw,circle,inner sep=1pt,fill=black,label=-90:{$v_i$}] {};
		\node (3) at (1,0.3) [draw,circle,inner sep=1.2pt,label=0:{}] {};
		\node (4) at (1,0.8) [draw,circle,inner sep=1.2pt,label=0:{}] {};
		\node (5) at (1,1.5) [draw,circle,inner sep=1.2pt,label=0:{}] {};
		\node (6) at (1,2) [draw,circle,inner sep=1.2pt,label=180:{}] {};
        \node (7) at (1,2.5) [draw,circle,inner sep=1.2pt,label=0:{}] {};
		\node (8) at (1,3) [draw,circle,inner sep=1.2pt,label=180:{}] {};
        \node (9) at (1,3.5) [draw,circle,inner sep=1.2pt,label=180:{}] {};
		
		\path
		(1) edge [-, black]  (2)
        (2) edge [->, black] (1,0)
		(4) edge [dash pattern=on 0.5pt off 3pt, line cap=round] (5);
		
		\coordinate (mid_v1_v2) at ($ (1)!0.25!(2) $);  % 3/4 arista v1-v2
		
		%Flechas
		\draw[->,thick,bend right=30, red] (2) to (8);    
		\draw[->,thick,bend right=30,red] (7) to (mid_v1_v2);  
		\draw[->,thick,bend right=50,red] (3) to (6);
        \draw[->,thick,bend right=50,red] (5) to (4);
        \draw[->,thick,bend left=30,red] (1) to (9);
	\end{tikzpicture}
\end{center}

\begin{remark}\label{rmk:sigmaij}
Let $(v_i,v_j)$ be an oriented edge.
If $\kappa_{i,j}$ is even, all basis elements corresponding to the edge \((v_i,v_j)\), but the head $v_j$, appear once in $\sigma_{i,j}$. 
While if $\kappa_{i,j}$ is odd, all basis elements corresponding to the edge \((v_i,v_j)\) appear once in $\sigma_{i,j}$. 
\end{remark}

\medskip

From now on, we restrict ourselves to consider three disjoint families of chain digraphs.
\begin{itemize}
    \item $\F_1$. Forests consisting of two rooted trees, directed from the leaves to the root, with an even chain function
                ($\kappa_{i,j}$ is even for all directed edges $(v_i,v_j)$).
    \item $\F_2$. Unicyclic graphs, directed centripetally, with an even chain function.
    \item $\F_3$. Rooted trees, directed from the leaves to the root, with an almost even chain function
                 ($\kappa_{i,j}$ is even on all directed edges $(v_i,v_j)$, but one ending in the root where it is odd).
\end{itemize}

\medskip

Let \(\F\) be the family of chain digraphs that are unions of chain digraphs taken from \(\F_1\cup\F_2\cup\F_3\). Since \(\g(\kappa,G)\) decomposes as a direct sum over the connected components of \(G\), and since the direct sum of symplectic Lie algebras is symplectic, it is enough to prove that \(\g(\kappa,G)\) is symplectic for every chain digraph in each of the families \(\F_1\), \(\F_2\), and \(\F_3\).

\medskip

Since the sum of closed 2-forms is closed, for every given digraph $(\kappa,G)$ with $G=(V,E)$,
the 2-form
\begin{equation}\label{eqn:sigmaE}
\sigma_E = \sum_{(v_i,v_j)\in E} \sigma_{i,j},
\end{equation}
is closed. Here, $\sigma_{i,j}$ is the 2-form defined in \eqref{eqn:sigma-even} or in \eqref{eqn:sigma-odd}
depending on the parity of $\kappa_{i,j}$.

%----------------------------------------------------------------------------------
\subsection{The family $\F_1$}

\begin{theorem}
Let $(\kappa,G)$ be a chain digraph of the family $\F_1$. 
Then, the corresponding Lie algebra $\g(\kappa,G)$ is symplectic.
Moreover, if $v_1$ and $v_2$ are the roots of the two trees forming $G$,
the 2-form
\[ \sigma_G= \sigma_E + v_1^*\wedge v_2^*, \]
where $\sigma_E$ is as in \eqref{eqn:sigmaE}, is a symplectic form on $\g(\kappa,G)$.
\end{theorem}

\begin{proof}
The form $\sigma_G$ is clearly closed.

Since $\kappa$ is even, by Remark \ref{rmk:sigmaij}, all the elements of the standard basis of $\g(\kappa,G)$, but $v_1$ and $v_2$, appear once in $\sigma_E$.
Hence, $\sigma_G$ is nondegenerate (see \S \ref{subsec:sym}).
\end{proof}

\begin{example}   
For the following chain digraph $G$ of $\F_1$,

\medskip

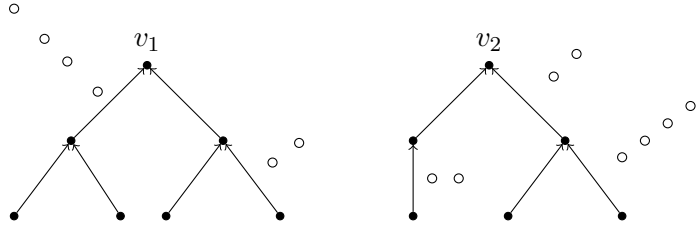
\begin{center}
\begin{tikzpicture}
\tikzset{every loop/.style={looseness=30}}
% Nodos
\node (1) at (0,0) [draw,circle,inner sep=1pt,fill=black!100,label=90:{$v_1$}] {};
\node (2) at (-1,-1) [draw,circle,inner sep=1pt,fill=black!100,label=90:{}] {};
\node (3) at (1,-1) [draw,circle,inner sep=1pt,fill=black!100,label=90:{}] {};
\node (4) at (-1.75,-2) [draw,circle,inner sep=1pt,fill=black!100,label=90:{}] {};
\node (5) at (-0.35,-2)  [draw,circle,inner sep=1pt,fill=black!100,label=180:{}] {};
\node (6) at (0.25,-2)  [draw,circle,inner sep=1pt,fill=black!100,label=180:{}] {};
\node (7) at (1.75,-2)  [draw,circle,inner sep=1pt,fill=black!100,label=0:{}] {};

\node (8) at (4.5,0)  [draw,circle,inner sep=1pt,fill=black!100,label=90:{$v_2$}] {};
\node (9) at (3.5,-1)  [draw,circle,inner sep=1pt,fill=black!100,label=180:{}] {};
\node (10) at (5.5,-1)  [draw,circle,inner sep=1pt,fill=black!100,label=0:{}] {};
\node (11) at (3.5,-2)  [draw,circle,inner sep=1pt,fill=black!100,label=180:{}] {};
\node (12) at (4.75,-2)  [draw,circle,inner sep=1pt,fill=black!100,label=0:{}] {};
\node (13) at (6.25,-2)  [draw,circle,inner sep=1pt,fill=black!100,label=0:{}] {};
% cadenas w_{12}^k
\node (14) at (-0.65,-0.35) [draw,circle,inner sep=1.2pt,label=0:{}] {};
\node (15) at (-1.05,0.05) [draw,circle,inner sep=1.2pt,label=0:{}] {};
\node (16) at (-1.35,0.35) [draw,circle,inner sep=1.2pt,label=0:{}] {};
\node (17) at (-1.75,0.75) [draw,circle,inner sep=1.2pt,label=0:{}] {};
% cadenas w_{37}^k
\node (18) at (1.65,-1.29) [draw,circle,inner sep=1.2pt,label=0:{}] {};
\node (19) at (2,-1.03) [draw,circle,inner sep=1.2pt,label=0:{}] {};
% cadenas w_{1013}^k
\node (20) at (6.25,-1.22) [draw,circle,inner sep=1.2pt,label=0:{}] {};
\node (21) at (6.55,-1) [draw,circle,inner sep=1.2pt,label=0:{}] {};
\node (22) at (6.85,-0.77) [draw,circle,inner sep=1.2pt,label=0:{}] {};
\node (23) at (7.15,-0.54) [draw,circle,inner sep=1.2pt,label=0:{}] {};
% cadenas w_{810}^k
\node (24) at (5.35,-0.15) [draw,circle,inner sep=1.2pt,label=0:{}] {};
\node (25) at (5.65,0.15) [draw,circle,inner sep=1.2pt,label=0:{}] {};
% cadenas w_{911}
\node (26) at (3.75,-1.5) [draw,circle,inner sep=1.2pt,label=0:{}] {};
\node (27) at (4.1,-1.5) [draw,circle,inner sep=1.2pt,label=0:{}] {};
\path
(1)  edge [<-, black] node[above=0] {} (2)
(1)  edge [<-, black] node[above=0] {} (3)
(2)  edge [<-, black] node[above=0] {} (4)
(2)  edge [<-, black] node[above=0] {} (5)
(3)  edge [<-, black] node[above=0] {} (6)
(3)  edge [<-, black] node[above=0] {} (7)
(8)  edge [<-, black] node[above=0] {} (9)
(8)  edge [<-, black] node[above=0] {} (10)
(9)  edge [<-, black] node[above=0] {} (11)
(10)  edge [<-, black] node[above=0] {} (12)
(10)  edge [<-, black] node[above=0] {} (13);
\end{tikzpicture}
\end{center}
the symplectic form $\sigma_G$ can be depicted as 
\begin{center}
\begin{tikzpicture}
\tikzset{every loop/.style={looseness=30}}
% Nodos
\node (1) at (0,0) [draw,circle,inner sep=1pt,fill=black!100,label=90:{$v_1$}] {};
\node (2) at (-1,-1) [draw,circle,inner sep=1pt,fill=black!100,label=90:{}] {};
\node (3) at (1,-1) [draw,circle,inner sep=1pt,fill=black!100,label=90:{}] {};
\node (4) at (-1.75,-2) [draw,circle,inner sep=1pt,fill=black!100,label=90:{}] {};
\node (5) at (-0.35,-2)  [draw,circle,inner sep=1pt,fill=black!100,label=180:{}] {};
\node (6) at (0.25,-2)  [draw,circle,inner sep=1pt,fill=black!100,label=180:{}] {};
\node (7) at (1.75,-2)  [draw,circle,inner sep=1pt,fill=black!100,label=0:{}] {};

\node (8) at (4.5,0)  [draw,circle,inner sep=1pt,fill=black!100,label=90:{$v_2$}] {};
\node (9) at (3.5,-1)  [draw,circle,inner sep=1pt,fill=black!100,label=180:{}] {};
\node (10) at (5.5,-1)  [draw,circle,inner sep=1pt,fill=black!100,label=0:{}] {};
\node (11) at (3.5,-2)  [draw,circle,inner sep=1pt,fill=black!100,label=180:{}] {};
\node (12) at (4.75,-2)  [draw,circle,inner sep=1pt,fill=black!100,label=0:{}] {};
\node (13) at (6.25,-2)  [draw,circle,inner sep=1pt,fill=black!100,label=0:{}] {};
% cadenas w_{12}^k
\node (14) at (-0.65,-0.35) [draw,circle,inner sep=1.2pt,label=0:{}] {};
\node (15) at (-1.05,0.05) [draw,circle,inner sep=1.2pt,label=0:{}] {};
\node (16) at (-1.35,0.35) [draw,circle,inner sep=1.2pt,label=0:{}] {};
\node (17) at (-1.75,0.75) [draw,circle,inner sep=1.2pt,label=0:{}] {};
% cadenas w_{37}^k
\node (18) at (1.65,-1.29) [draw,circle,inner sep=1.2pt,label=0:{}] {};
\node (19) at (2,-1.03) [draw,circle,inner sep=1.2pt,label=0:{}] {};
% cadenas w_{1013}^k
\node (20) at (6.25,-1.22) [draw,circle,inner sep=1.2pt,label=0:{}] {};
\node (21) at (6.55,-1) [draw,circle,inner sep=1.2pt,label=0:{}] {};
\node (22) at (6.85,-0.77) [draw,circle,inner sep=1.2pt,label=0:{}] {};
\node (23) at (7.15,-0.54) [draw,circle,inner sep=1.2pt,label=0:{}] {};
% cadenas w_{810}^k
\node (24) at (5.35,-0.15) [draw,circle,inner sep=1.2pt,label=0:{}] {};
\node (25) at (5.65,0.15) [draw,circle,inner sep=1.2pt,label=0:{}] {};
% cadenas w_{911}
\node (26) at (3.75,-1.5) [draw,circle,inner sep=1.2pt,label=0:{}] {};
\node (27) at (4.1,-1.5) [draw,circle,inner sep=1.2pt,label=0:{}] {};
\path
(1)  edge [<-, black] node[above=0] {} (2)
(1)  edge [<-, black] node[above=0] {} (3)
(2)  edge [<-, black] node[above=0] {} (4)
(2)  edge [<-, black] node[above=0] {} (5)
(3)  edge [<-, black] node[above=0] {} (6)
(3)  edge [<-, black] node[above=0] {} (7)
(8)  edge [<-, black] node[above=0] {} (9)
(8)  edge [<-, black] node[above=0] {} (10)
(9)  edge [<-, black] node[above=0] {} (11)
(10)  edge [<-, black] node[above=0] {} (12)
(10)  edge [<-, black] node[above=0] {} (13);

% --- Puntos medios de las aristas ---
\coordinate (mid_v2_v1) at ($ (1)!0.65!(2) $);  % arista v2-v1
\coordinate (mid_v3_v1) at ($ (1)!0.35!(3) $);  % arista v3-v1
\coordinate (mid_v7_v3) at ($ (3)!0.75!(7) $);  % arista v7-v3
\coordinate (mid_v6_v3) at ($ (3)!0.35!(6) $);  % arista v6-v3
\coordinate (mid_v4_v2) at ($ (2)!0.35!(4) $);  % arista v4-v2
\coordinate (mid_v5_v2) at ($ (2)!0.35!(5) $);  % arista v5-v2
\coordinate (mid_v8_v9) at ($ (8)!0.35!(9) $);  % arista v8-v9
\coordinate (mid_v8_v10) at ($ (8)!0.65!(10) $);  % arista v8-v10
\coordinate (mid_v9_v11) at ($ (9)!0.65!(11) $);  % arista v11-v10
\coordinate (mid_v10_v12) at ($ (10)!0.35!(12) $);  % arista v11-v10
\coordinate (mid_v10_v13) at ($ (10)!0.65!(13) $);  % arista v12-v10

% --- Flechas curvas ---
\draw[->,thick,bend left=35,red] (1) to (8);             % v1 → v8
\draw[->,thick,bend left=55,red] (2) to (17);    % v2 → arista v2-v1
\draw[->,thick,bend right=55, red] (3) to (mid_v3_v1);     % v3 → arista v3-v1
\draw[->,thick,bend left=55,red] (4) to (mid_v4_v2);     % v4 → arista v4-v2
\draw[->,thick,bend right=55,red] (7) to (19);    % v7 → arista v7-v3
\draw[->,thick,bend left=55,red] (6) to (mid_v6_v3);    % v6 → arista v6-v3
\draw[->,thick,bend right=55,red] (5) to (mid_v5_v2);    % v6 → arista v6-v3
\draw[->,thick,bend left=55,red] (9) to (mid_v8_v9); 
\draw[->,thick,bend right=55,red] (10) to (25);
\draw[->,thick,bend right=55,red] (11) to (27);
\draw[->,thick,bend left=55,red] (12) to (mid_v10_v12);
\draw[->,thick,bend right=55,red] (13) to (23);
\draw[->,thick,bend right=65,red] (14) to (15);
\draw[->,thick,bend right=65,red] (16) to (mid_v2_v1);
\draw[->,thick,bend left=65,red] (18) to (mid_v7_v3);
\draw[->,thick,bend left=65,red] (20) to (21);
\draw[->,thick,bend left=65,red] (22) to (mid_v10_v13);
\draw[->,thick,bend left=65,red] (24) to (mid_v8_v10);
\draw[->,thick,bend left=65,red] (26) to (mid_v9_v11);

\end{tikzpicture}
\end{center}
\end{example}

%----------------------------------------------------------------------------------
\subsection{The family $\F_2$}

\begin{theorem}
Let $(\kappa,G)$ be a chain digraph of the family $\F_2$. 
Then, the corresponding Lie algebra $\g(\kappa,G)$ is symplectic.
Moreover, the 2-form
\[ \sigma_G= \sigma_E, \]
where $\sigma_E$ is as in \eqref{eqn:sigmaE}, is a symplectic form on $\g(\kappa,G)$.
\end{theorem}

\begin{proof}
The form $\sigma_E$ is clearly closed and since $\kappa$ is even,
Remark \ref{rmk:sigmaij} implies that $\sigma_G$ is nondegenerate (see \S \ref{subsec:sym}).
\end{proof}

\begin{example}
For the following chain digraph $G$ of $\F_2$,

\medskip

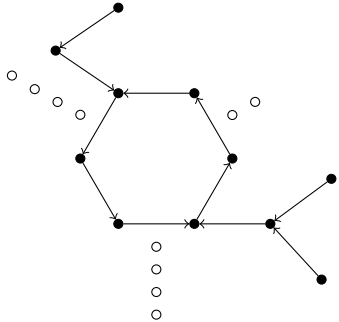
\begin{center}
\begin{tikzpicture}
\tikzset{every loop/.style={looseness=30}}
\node (1) at (1,0) [draw,circle,inner sep=1.2pt,fill=black!100,label=0:{}] {};
\node (2) at (0.5,{sqrt(3)/2}) [draw,circle,inner sep=1.2pt,fill=black!100,label=90:{}] {};
\node (3) at (-0.5,{sqrt(3)/2}) [draw,circle,inner sep=1.2pt,fill=black!100,label=90:{}] {};
\node (4) at (-1,0) [draw,circle,inner sep=1.2pt,fill=black!100,label=180:{}] {};
\node (5) at (-0.5,{-sqrt(3)/2}) [draw,circle,inner sep=1.2pt,fill=black!100,label=270:{}] {};
\node (6) at (0.5,{-sqrt(3)/2}) [draw,circle,inner sep=1.2pt,fill=black!100,label=270:{}] {};
\node (7) at (1.5,{-sqrt(3)/2}) [draw,circle,inner sep=1.2pt,fill=black!100,label=270:{}] {};
\node (8) at (2.303,-0.269) [draw,circle,inner sep=1.2pt,fill=black!100,label=270:{}] {};
\node (9) at (2.174,-1.605) [draw,circle,inner sep=1.2pt,fill=black!100,label=270:{}] {};
\node (10) at (-1.325,1.431) [draw,circle,inner sep=1.2pt,fill=black!100,label=270:{}] {};
\node (11) at (-0.5,2) [draw,circle,inner sep=1.2pt,fill=black!100,label=270:{}] {};
% w_{ij}^k
\node (12) at (0,-1.17) [draw,circle,inner sep=1.2pt,label=0:{}] {};
\node (13) at (0,-1.47) [draw,circle,inner sep=1.2pt,label=0:{}] {};
\node (14) at (0,-1.77) [draw,circle,inner sep=1.2pt,label=0:{}] {};
\node (15) at (0,-2.07) [draw,circle,inner sep=1.2pt,label=0:{}] {};
% w_{ij}^k
\node (16) at (1,0.577) [draw,circle,inner sep=1.2pt,label=0:{}] {};
\node (17) at (1.3,0.751) [draw,circle,inner sep=1.2pt,label=0:{}] {};
% w_{ij}^k
\node (18) at (-1,0.58) [draw,circle,inner sep=1.2pt,label=0:{}] {};
\node (19) at (-1.3,0.75) [draw,circle,inner sep=1.2pt,label=0:{}] {};
\node (20) at (-1.6,0.92) [draw,circle,inner sep=1.2pt,label=0:{}] {};
\node (21) at (-1.9,1.1) [draw,circle,inner sep=1.2pt,label=0:{}] {};

\path
(1) edge[->] (2)
(2) edge[->] (3)
(3) edge[->] (4)
(4) edge[->] (5)
(6) edge[->] (1)
(5) edge[->] (6)
(6) edge[<-] (7)
(7) edge[<-] (8)
(7) edge[<-] (9)
(3) edge[<-] (10)
(10) edge[<-] (11);

\end{tikzpicture}
\end{center}
the symplectic form $\sigma_{G}$ can be depicted as 
\begin{center}
\begin{tikzpicture}
\tikzset{every loop/.style={looseness=30}}
\node (1) at (1,0) [draw,circle,inner sep=1.2pt,fill=black!100,label=0:{}] {};
\node (2) at (0.5,{sqrt(3)/2}) [draw,circle,inner sep=1.2pt,fill=black!100,label=90:{}] {};
\node (3) at (-0.5,{sqrt(3)/2}) [draw,circle,inner sep=1.2pt,fill=black!100,label=90:{}] {};
\node (4) at (-1,0) [draw,circle,inner sep=1.2pt,fill=black!100,label=180:{}] {};
\node (5) at (-0.5,{-sqrt(3)/2}) [draw,circle,inner sep=1.2pt,fill=black!100,label=270:{}] {};
\node (6) at (0.5,{-sqrt(3)/2}) [draw,circle,inner sep=1.2pt,fill=black!100,label=270:{}] {};
\node (7) at (1.5,{-sqrt(3)/2}) [draw,circle,inner sep=1.2pt,fill=black!100,label=270:{}] {};
\node (8) at (2.303,-0.269) [draw,circle,inner sep=1.2pt,fill=black!100,label=270:{}] {};
\node (9) at (2.174,-1.605) [draw,circle,inner sep=1.2pt,fill=black!100,label=270:{}] {};
\node (10) at (-1.325,1.431) [draw,circle,inner sep=1.2pt,fill=black!100,label=270:{}] {};
\node (11) at (-0.5,2) [draw,circle,inner sep=1.2pt,fill=black!100,label=270:{}] {};
% w_{ij}^k
\node (12) at (0,-1.17) [draw,circle,inner sep=1.2pt,label=0:{}] {};
\node (13) at (0,-1.47) [draw,circle,inner sep=1.2pt,label=0:{}] {};
\node (14) at (0,-1.77) [draw,circle,inner sep=1.2pt,label=0:{}] {};
\node (15) at (0,-2.07) [draw,circle,inner sep=1.2pt,label=0:{}] {};
% w_{ij}^k
\node (16) at (1,0.577) [draw,circle,inner sep=1.2pt,label=0:{}] {};
\node (17) at (1.3,0.751) [draw,circle,inner sep=1.2pt,label=0:{}] {};
% w_{ij}^k
\node (18) at (-1,0.58) [draw,circle,inner sep=1.2pt,label=0:{}] {};
\node (19) at (-1.3,0.75) [draw,circle,inner sep=1.2pt,label=0:{}] {};
\node (20) at (-1.6,0.92) [draw,circle,inner sep=1.2pt,label=0:{}] {};
\node (21) at (-1.9,1.1) [draw,circle,inner sep=1.2pt,label=0:{}] {};
%Puntos medios
\coordinate (mid_v1_v2) at ($ (1)!0.35!(2) $);
\coordinate (mid_v2_v3) at ($ (2)!0.65!(3) $);
\coordinate (mid_v3_v4) at ($ (3)!0.65!(4) $);
\coordinate (mid_v4_v5) at ($ (4)!0.65!(5) $);
\coordinate (mid_v5_v6) at ($ (5)!0.35!(6) $);
\coordinate (mid_v6_v1) at ($ (6)!0.65!(1) $);
\coordinate (mid_v6_v7) at ($ (6)!0.35!(7) $);
\coordinate (mid_v7_v8) at ($ (7)!0.45!(8) $);
\coordinate (mid_v7_v9) at ($ (7)!0.45!(9) $);
\coordinate (mid_v3_v10) at ($ (3)!0.35!(10) $);
\coordinate (mid_v10_v11) at ($ (10)!0.35!(11) $);

\path
(1) edge[->] (2)
(2) edge[->] (3)
(3) edge[->] (4)
(4) edge[->] (5)
(6) edge[->] (1)
(5) edge[->] (6)
(6) edge[<-] (7)
(7) edge[<-] (8)
(7) edge[<-] (9)
(3) edge[<-] (10)
(10) edge[<-] (11);

% --- Flechas curvas ---
\draw[->,thick,bend right=55,red] (1) to (17);
\draw[->,thick,bend right=55,red] (16) to (mid_v1_v2);
\draw[->,thick,bend right=55,red] (2) to (mid_v2_v3);
\draw[->,thick,bend right=5,red] (3) to (21);
\draw[->,thick,bend right=55,red] (4) to (mid_v4_v5);
\draw[->,thick,bend right=55,red] (5) to (15);
\draw[->,thick,bend right=55,red] (6) to (mid_v6_v1);
\draw[->,thick,bend left=55,red] (7) to (mid_v6_v7);
\draw[->,thick,bend right=55,red] (8) to (mid_v7_v8);
\draw[->,thick,bend left=55,red] (9) to (mid_v7_v9);
\draw[->,thick,bend left=55,red] (10) to (mid_v3_v10);
\draw[->,thick,bend right=55,red] (11) to (mid_v10_v11);
\draw[->,thick,bend right=75,red] (12) to (13);
\draw[->,thick,bend left=75,red] (14) to (mid_v5_v6);
\draw[->,thick,bend left=75,red] (18) to (19);
\draw[->,thick,bend right=75,red] (20) to (mid_v3_v4);

\end{tikzpicture}
\end{center}
\end{example}

%--------------------------------------------------------------------------------
\subsection{The family $\F_3$}

\begin{theorem}
Let $(\kappa,G)$ be a chain digraph of the family $\F_3$. 
Then, the corresponding Lie algebra $\g(\kappa,G)$ is symplectic.
Moreover, the 2-form
\[ \sigma_G= \sigma_E, \]
where $\sigma_E$ is as in \eqref{eqn:sigmaE}, is a symplectic form on $\g(\kappa,G)$.
\end{theorem}

\begin{proof}
The form $\sigma_E$ is clearly closed. 
Also, every vertex, except for the root, is the tail of exactly one edge. Since the root is the head of an odd edge, it follows from Remark \ref{rmk:sigmaij} that all the elements of the standard basis of \(\g(\kappa,G)\) appear exactly once in $\sigma_G$. 
Consequently, $\sigma_G$ is nondegenerate (see \S \ref{subsec:sym}).
\end{proof}

\begin{example}
For the following chain digraph $G$ of $\F_3$,

\medskip

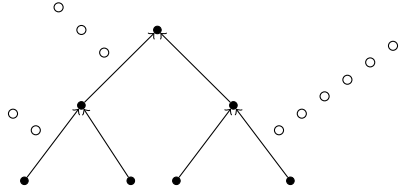
\begin{center}
\begin{tikzpicture}
\tikzset{every loop/.style={looseness=30}}

% Nodos principales
\node (1) at (0,0) [draw,circle,inner sep=1pt,fill=black!100,label=90:{}] {};
\node (2) at (-1,-1) [draw,circle,inner sep=1pt,fill=black!100] {};
\node (3) at (1,-1) [draw,circle,inner sep=1pt,fill=black!100] {};
\node (4) at (-1.75,-2) [draw,circle,inner sep=1pt,fill=black!100] {};
\node (5) at (-0.35,-2)  [draw,circle,inner sep=1pt,fill=black!100] {};
\node (6) at (0.25,-2)  [draw,circle,inner sep=1pt,fill=black!100] {};
\node (7) at (1.75,-2)  [draw,circle,inner sep=1pt,fill=black!100] {};

% cadena diagonal entre 1 y 2
\node (8) at (-0.7,-0.3) [draw,circle,inner sep=1.2pt,label=0:{}] {};
\node (9) at (-1,0) [draw,circle,inner sep=1.2pt,label=0:{}] {};
\node (10) at (-1.3,0.3) [draw,circle,inner sep=1.2pt,label=0:{}] {};
%\node (11) at (-1.6,0.6) [draw,circle,inner sep=1.2pt,label=0:{}] {};

% cadena diagonal entre 3 y 7 (un poco más arriba)
\node (12) at (1.6,-1.33) [draw,circle,inner sep=1.2pt,label=0:{}] {};
\node (13) at (1.9,-1.11) [draw,circle,inner sep=1.2pt,label=0:{}] {};
\node (14) at (2.2,-0.88) [draw,circle,inner sep=1.2pt] {};
\node (15) at (2.5,-0.66) [draw,circle,inner sep=1.2pt] {};
\node (16) at (2.8,-0.43) [draw,circle,inner sep=1.2pt] {};
\node (17) at (3.1,-0.21) [draw,circle,inner sep=1.2pt] {};

% cadena diagonal entre 2 y 4 (un poco más arriba)
\node (18) at (-1.6,-1.33) [draw,circle,inner sep=1.2pt] {};
\node (19) at (-1.9,-1.11) [draw,circle,inner sep=1.2pt] {};

\path
(1) edge [<-] (2)
(1) edge [<-] (3)
(2) edge [<-] (4)
(2) edge [<-] (5)
(3) edge [<-] (6)
(3) edge [<-] (7);

\end{tikzpicture}
\end{center}
the symplectic form $\sigma_G$ can be depicted as 
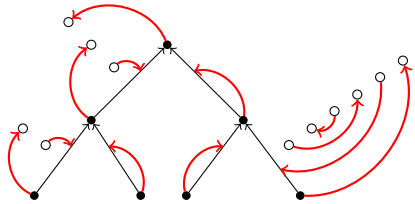
\begin{center}
\begin{tikzpicture}
\tikzset{every loop/.style={looseness=30}}

% Nodos principales
\node (1) at (0,0) [draw,circle,inner sep=1pt,fill=black!100,label=90:{}] {};
\node (2) at (-1,-1) [draw,circle,inner sep=1pt,fill=black!100] {};
\node (3) at (1,-1) [draw,circle,inner sep=1pt,fill=black!100] {};
\node (4) at (-1.75,-2) [draw,circle,inner sep=1pt,fill=black!100] {};
\node (5) at (-0.35,-2)  [draw,circle,inner sep=1pt,fill=black!100] {};
\node (6) at (0.25,-2)  [draw,circle,inner sep=1pt,fill=black!100] {};
\node (7) at (1.75,-2)  [draw,circle,inner sep=1pt,fill=black!100] {};

% cadena diagonal entre 1 y 2
\node (8) at (-0.7,-0.3) [draw,circle,inner sep=1.2pt,label=0:{}] {};
\node (9) at (-1,0) [draw,circle,inner sep=1.2pt,label=0:{}] {};
\node (10) at (-1.3,0.3) [draw,circle,inner sep=1.2pt,label=0:{}] {};
%\node (11) at (-1.6,0.6) [draw,circle,inner sep=1.2pt,label=0:{}] {};

% cadena diagonal entre 3 y 7 (un poco más arriba)
\node (12) at (1.6,-1.33) [draw,circle,inner sep=1.2pt,label=0:{}] {};
\node (13) at (1.9,-1.11) [draw,circle,inner sep=1.2pt,label=0:{}] {};
\node (14) at (2.2,-0.88) [draw,circle,inner sep=1.2pt] {};
\node (15) at (2.5,-0.66) [draw,circle,inner sep=1.2pt] {};
\node (16) at (2.8,-0.43) [draw,circle,inner sep=1.2pt] {};
\node (17) at (3.1,-0.21) [draw,circle,inner sep=1.2pt] {};

% cadena diagonal entre 2 y 4 (un poco más arriba)
\node (18) at (-1.6,-1.33) [draw,circle,inner sep=1.2pt] {};
\node (19) at (-1.9,-1.11) [draw,circle,inner sep=1.2pt] {};

%Puntos medios
\coordinate (mid_v1_v2) at ($ (1)!0.35!(2) $);
\coordinate (mid_v1_v3) at ($ (1)!0.35!(3) $);
\coordinate (mid_v2_v4) at ($ (2)!0.35!(4) $);
\coordinate (mid_v2_v5) at ($ (2)!0.35!(5) $);
\coordinate (mid_v3_v6) at ($ (3)!0.35!(6) $);
\coordinate (mid_v3_v7) at ($ (3)!0.65!(7) $);

\path
(1) edge [<-] (2)
(1) edge [<-] (3)
(2) edge [<-] (4)
(2) edge [<-] (5)
(3) edge [<-] (6)
(3) edge [<-] (7);

% --- Flechas curvas ---
\draw[->,thick,bend left=55,red] (8) to (mid_v1_v2);
\draw[->,thick,bend right=55,red] (3) to (mid_v1_v3);
\draw[->,thick,bend left=55,red] (18) to (mid_v2_v4);
\draw[->,thick,bend right=55,red] (5) to (mid_v2_v5);
\draw[->,thick,bend left=55,red] (6) to (mid_v3_v6);
\draw[->,thick,bend left=55,red] (16) to (mid_v3_v7);
\draw[->,thick,bend right=55,red] (1) to (10);
\draw[->,thick,bend left=55,red] (2) to (9);
\draw[->,thick,bend left=55,red] (4) to (19);
\draw[->,thick,bend right=55,red] (7) to (17);
\draw[->,thick,bend right=55,red] (12) to (15);
\draw[->,thick,bend left=55,red] (14) to (13);

\end{tikzpicture}
\end{center}

\end{example}

%------------------------------------
\subsection{The almost abelian case}

An $n$-dimensional almost abelian Lie algebra $\g$, over $\K$, is of the form $\g=T\ltimes\K^{n-1}$,
where $T$ can be assumed to be in Jordan form.

Following \cite{ABDGH}, by replacing $k$ by $q$,
\[ T=J_{n_1}^{p_1} \oplus \dots \oplus J_{n_q}^{p_q} \oplus J_1^t. \]
Here, $J_{n_i}$ is the Jordan block
\[ \begin{pmatrix} 0 & 1 & & & \\ 
                     & 0 & 1 & & \\
                     & & 0 & \ddots & \\
                     & & & \ddots & 1 \\
                     & & & & 0
   \end{pmatrix}.
\]
Note that the dimension of $\g$ is $\dim\g=n=p_1 n_1+\dots+p_q n_q+t$.

It is not difficult to see that $\g=\g(\kappa,G)$ where the digraph $G$ is the one below, and $\kappa_i=n_i+2$.
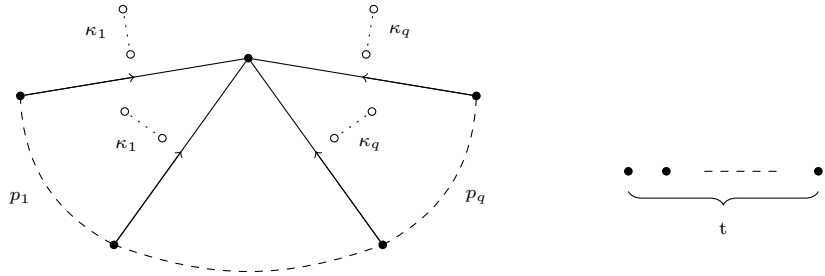
\begin{center}
\begin{tikzpicture}

% Nodos principales
\node (1) at (0,0) [draw,circle,inner sep=1pt,fill=black!100,label=90:{}] {};

\coordinate (A) at (-3,-0.5);
\node (2) at (A) [draw,circle,inner sep=1pt,fill=black!100,label=90:{}] {};
\coordinate (B) at ([rotate around={45:(0,0)}] A);
\node (3) at (B) [draw,circle,inner sep=1pt,fill=black!100,label=90:{}] {};

\coordinate (C) at (3,-0.5);
\node (4) at (C) [draw,circle,inner sep=1pt,fill=black!100,label=90:{}] {};
\coordinate (D) at ([rotate around={-45:(0,0)}] C);
\node (5) at (D) [draw,circle,inner sep=1pt,fill=black!100,label=90:{}] {};

\coordinate (M1) at ($(0,0)!0.5!(A)$);
\coordinate (M2) at ($(0,0)!0.5!(B)$);
\coordinate (M3) at ($(0,0)!0.5!(C)$);
\coordinate (M4) at ($(0,0)!0.5!(D)$);

\path
(1) edge [-] (2)
(2) edge [->] (M1)
(1) edge [-] (3)
(3) edge [->] (M2)
(1) edge [-] (4)
(4) edge [->] (M3)
(1) edge [-] (5)
(5) edge [->] (M4);

\coordinate (K1) at ($(M1)+0.1*(-0.5,3)$);
\coordinate (K2) at ($(M1)+0.3*(-0.5,3)$);
\path (K1) edge [dash pattern=on 0.5pt off 3pt, line cap=round] (K2);
\node (10) at (K1) [draw,circle,inner sep=1pt,fill=black!0,label=90:{}] {};
\node (11) at (K2) [draw,circle,inner sep=1pt,fill=black!0,label=90:{}] {};

\node (12) at (-2,0) [label=90:{\tiny{$\kappa_1$}}]{};

\coordinate (K3) at ([rotate around={45:(0,0)}] K1);
\coordinate (K4) at ([rotate around={45:(0,0)}] K2);
\path (K3) edge [dash pattern=on 0.5pt off 3pt, line cap=round] (K4);
\node (13) at (K3) [draw,circle,inner sep=1pt,fill=black!0,label=90:{}] {};
\node (14) at (K4) [draw,circle,inner sep=1pt,fill=black!0,label=90:{}] {};

\node (15) at (-1.6,-1.5) [label=90:{\tiny{$\kappa_1$}}]{};

\coordinate (K5) at ($(M3)+0.1*(0.5,3)$);
\coordinate (K6) at ($(M3)+0.3*(0.5,3)$);
\path (K5) edge [dash pattern=on 0.5pt off 3pt, line cap=round] (K6);
\node (20) at (K5) [draw,circle,inner sep=1pt,fill=black!0,label=90:{}] {};
\node (21) at (K6) [draw,circle,inner sep=1pt,fill=black!0,label=90:{}] {};

\node (22) at (2,0) [label=90:{\tiny{$\kappa_q$}}]{};

\coordinate (K7) at ([rotate around={-45:(0,0)}] K5);
\coordinate (K8) at ([rotate around={-45:(0,0)}] K6);
\path (K7) edge [dash pattern=on 0.5pt off 3pt, line cap=round] (K8);
\node (23) at (K7) [draw,circle,inner sep=1pt,fill=black!0,label=90:{}] {};
\node (24) at (K8) [draw,circle,inner sep=1pt,fill=black!0,label=90:{}] {};

\node (15) at (1.6,-1.5) [label=90:{\tiny{$\kappa_q$}}]{};

\draw[dashed] (A) to[bend right=30] (B);
\draw[dashed] (B) to[bend right=20] (D);
\draw[dashed] (D) to[bend right=30] (C);

\node (31) at (-3,-2.2) [label=90:{\tiny{$p_1$}}]{};
\node (32) at (3,-2.2) [label=90:{\tiny{$p_q$}}]{};

\node (41) at (5,-1.5) [draw,circle,inner sep=1pt,fill=black!100,label=90:{}] {};
\node (42) at (5.5,-1.5) [draw,circle,inner sep=1pt,fill=black!100,label=90:{}] {};
\node (43) at (7.5,-1.5) [draw,circle,inner sep=1pt,fill=black!100,label=90:{}] {};
\draw[dashed] (6,-1.5) to (7,-1.5);
\draw[decorate,decoration={brace,mirror,amplitude=5pt}]
([yshift=-6pt]41.south) -- ([yshift=-6pt]43.south)
node[midway,yshift=-14pt] {$\text{\tiny t}$};

\end{tikzpicture}
\end{center}

If $t$ is even and $\kappa$ is almost even, or $t$ is odd and $\kappa$ is even, 
then it follows from Theorems \ref{thm:even} and \ref{thm:odd}
that $\g=\g(\kappa,G)$ is symplectic.
Moreover, an explicit symplectic form is easy to write down.

In \cite{ABDGH} the authors give necessary and sufficient conditions for $\g=\g(\kappa,G)$ to be symplectic,
that are weaker than the ones described by us for this family.

%=============================================================================
\section{Symplectic Lie algebras of prescribed nilpotency type}

Given a finite-dimensional $k$-step nilpotent Lie algebra $\g$, its \emph{nilpotency type} is the $k+1$-tuple 
$(a_1, a_2, \dots, a_{k+1})$ of dimensions of the successive quotients of the lower central series, i.e.,
\[
a_i = \dim(\mathfrak{g}^{i} / \mathfrak{g}^{i+1}).
\]
Observe that \(a_{k+1}=0\).

\begin{lemma}\label{lm:typ}
Given a chain digraph \((\kappa,G)\), the type $(a_1,\dots,a_{k+1})$ of the Lie algebra $\g(\kappa,G)$ satisfies
\[
|\kappa^{-1}(\{l\})|=a_{l}-a_{l+1},
\]
 for all \(2\leq l\leq k\).
\end{lemma}

\begin{proof}
It follows from Proposition \ref{prop:bracket} that
\[a_l=|\bigcup_{i=l}^k B_{i}^\kappa|-|\bigcup_{i=l+1}^k B_{i}^\kappa|=|B_l^\kappa|. \] 
Hence, for each $2\le l\le k-1$, 
\begin{align*}
a_l-a_{l+1}&=|B_l^\kappa|-|B_{l+1}^\kappa|\\
&=|\{w_{i,j}^l: (v_i,v_j)\in E,\; l\leq \kappa_{i,j} \}|-|\{w_{i,j}^{l+1}: (v_i,v_j)\in E,\;l+1\leq \kappa_{i,j} \}|\\
&=|\{(v_i,v_j)\in E :\, l\leq \kappa_{i,j} \}|-|\{(v_i,v_j)\in E :\, l+1\leq \kappa_{i,j} \}|\\
&=|\{(v_i,v_j)\in E :\, l= \kappa_{i,j} \}|\\
&=|\kappa^{-1}(\{l\})|.
\end{align*}
For $l=k$, $$a_k-a_{k+1}=a_k=|B_k^\kappa|=|\kappa^{-1}(\{k\})|.$$
\end{proof}

\begin{theorem}\label{thm:type}
Given a $k+1$-tuple of natural numbers $(a_1, a_2, \dots, a_k,0)$, let 
\[ S=\sum_{\substack{i\ge 3 \\ i \text{ odd}}}a_i - \sum_{\substack{i\ge 4 \\ i \text{ even}}}a_i. \] 
If $a_1+\dots + a_k$ is even, $a_i \geq a_{i+1}$, for all $i=1,\dots,k-1$, and $a_1-a_2\geq S$,
then there exists a symplectic $k$-step nilpotent Lie algebra of type \((a_1,\dots,a_k,0)\).
\end{theorem}

\begin{proof}
Let $(a_1, a_2, \dots, a_k,0)$ be a sequence satisfying the hypothesis of the theorem, \begin{comment} we give ``steps" to construct a family  \(\mathcal{F}_{(a_1,\dots,a_k)}\) of chain digraphs such that the type of \(\g(\kappa,G)\) is \((a_1,\dots,a_k)\), for all \((\kappa,G)\in\F_{(a_1,\dots,a_k,0)}\). Also, \(\mathcal{F}_{(a_1,\dots,a_k)}\) is a subfamily of \(\mathcal{F}\), described in Section \ref{sec:fam-sym}. So, if \((\kappa,G)\in\mathcal{F}_{(a_1,\dots,a_k)}\), the Lie algebra \(\g(\kappa,G)\) is symplectic.  \end{comment}
the family of chain digraphs \(\mathcal{F}_{(a_1,\dots,a_k,0)}\) is formed by the chain digraphs that can be constructed following the steps:

\begin{enumerate}
\item For each \(2\leq i\leq k\), construct \(a_{i}-a_{i+1}\) edges with \(i-2\) over points. Observe that the number of edges that has an odd number of over points is \(S\). 
\item Construct some boxes as follows:
\begin{enumerate}[(a)]
\item In \(S\) disjoint boxes, \(A_1,\dots, A_{S}\), put one and only one edge with an odd number of over points. 
\item In \(T=a_1-a_2-S\) disjoint boxes, \(B_{1},\dots, B_{T}\), put one single vertex on each one. 
\item Adjoint one empty box, \(C\). 
\end{enumerate}
\item We have already located in boxes only the odd edges. Distribute the remaining edges (the edges with an even number of over points) in the boxes (there are \(S+T+1\) boxes). The only restriction to distribute it is: if the box \(C\) is non-empty, then it has at least 3 edges.  
\item Glue the vertices that are in the same box to form a graph and direct it as follows:
\begin{enumerate}[(a)]
\item In each box \(A_i\) form a tree. Then, choose one of the vertices of the odd edge to be the root, and direct it from the leaves to the root.   
\item In each box \(B_i\) form a tree. Then, choose a root and direct it from the leaves to the root.
\item In the box \(C\) form a graph that is the union of unicyclic graphs (the existence of at least one such graph is guaranteed by the presence of at least three edges). Then, direct each connected component centripetally.
\end{enumerate}
\item Join the digraphs of all boxes and define its chain function at an edge equal to the number of its over points plus 2.  
\end{enumerate}

Now observe that since \(a_1+\dots+a_k\) is even, then 
\[
a_1+\sum_{\substack{i\ge 3 \\ i \text{ odd}}} a_{i} \quad \text{ and } \quad a_2+\sum_{\substack{i\ge 3 \\ i \text{ odd}}} a_{i}
\] 
have the same parity. Thus, $T$ is even.

Each chain digraph in the family \(\mathcal{F}_{(a_1,\dots,a_k,0)}\) is a union of chain digraphs in the family \(\mathcal{F}_3\) (in the boxes \(A_i\)), in the family \(\mathcal{F}_1\) (in the boxes \(B_i\), since \(T\) is even) and in the family \(\mathcal{F}_2\) (in the box \(C\)) (see Section \ref{sec:fam-sym}). Therefore, \(\mathcal{F}_{(a_1,\dots,a_k,0)}\) is a subfamily of  \(\mathcal{F}\) and if \((\kappa,G)\in  \mathcal{F}_{(a_1,\dots,a_k,0)}\), then \(\g(\kappa,G)\) is symplectic.

Let \((\kappa,G)\) be a chain digraph that can be formed following the previous steps. From Steps (1) and (5), the maximum value of \(\kappa\) is \(k\), so \(\g(\kappa,G)\) is \(k\)-step nilpotent. It remains to prove that its type is \((a_1,\dots,a_k,0)\). 

Let \((n_1,\dots,n_k,0)\) be the type of \(\g(\kappa,G)\), from Lemma \ref{lm:typ} and Step (1), we have 
\begin{align*}
n_i-n_{i+1}&=|\kappa^{-1}(\{i\})|\\
&=\{\text{edges with \(i-2\) over points}\}\\
&=a_i-a_{i+1},
\end{align*}
for all \(2\leq i\leq k\). 
In particular, it satisfies \(n_k=a_k\). So, inductively, we have \(n_i=a_i\) for all \(2\leq i\leq k\).

It only remains to prove that \(n_1=a_1\). Observe that \(n_1\) is the number of vertices in the graph \(G\). Also, 
\[G=A\cup B\cup C,\]
where \(A=(V_A,E_A)\) is the disjoint union of the graphs in the boxes \(A_i\), \(B=(V_B,E_B)\) is the union of the graphs in the boxes \(B_i\), and \(C\) is the graph in the box \(C=(V_C,E_C)\).  
Since \(A\) and \(B\) are forests and \(C\) is a union of unicyclics graphs, we have that

\begin{align*}
|V_A|&=|E_A|+|\{\text{connected components of } A\}|;\\
|V_B|&=|E_B|+|\{\text{connected components of } B\}|;\\
|V_C|&=|E_C|.
\end{align*}
From step (2), we have
\begin{align*}
|V_A|&=|E_A|+S;\\
|V_B|&=|E_B|+T;\\
|V_C|&=|E_C|.
\end{align*}
So, 
\begin{align*}
|V_G|&=|V_A|+|V_B|+|V_C|\\
&=|E_A|+S + |E_B|+ T +|E_C|\\
&=|E_A|+S+|E_B|+a_1-a_2-S+|E_C|\\
&=|E_A|+|E_B|+|E_C|+a_1-a_2.
\end{align*}
We already have proved that \(a_2=n_2\), and we also know that \(n_2=|E_G|=|E_A|+|E_B|+|E_C|\). Thus, 
\[|V_G|=n_2+a_1-a_2=a_1.\]

Therefore, for every chain digraph \((\kappa,G)\in\F_{(a_1,\dots,a_k,0)}\), the Lie algebra $\g(\kappa,G)$ is symplectic and its type is \((a_1,\dots,a_k,0)\).
\end{proof}

\begin{remark}
It is not difficult to prove that if \((\kappa,G)\in\mathcal{F}\) (the family of chain digraphs described in Section \ref{sec:fam-sym}), then the type of \(\g(\kappa,G)\) is a sequence that satisfies the hypothesis of Theorem \ref{thm:type}.
\end{remark}

\begin{example}
 Given the sequence of numbers $(6,5,3,3,2,1,0)$, 
 which satisfies the conditions of the previous theorem, we construct a particular chain digraph \((\kappa,G)\)  such that the Lie algebra \(\g(\kappa,G)\) is symplectic and its type is $(6,5,3,3,2,1,0)$.
 
 First,
\[S=a_3+a_5-(a_4+a_6)=1 \quad \text{ and } \quad T=a_1-a_2-S=0.\]

So, there is 1 box of type \(A_i\), 0 boxes of type \(B_i\) and the box \(C\).

Compute the differences $a_i - a_{i+1}$ for all $2\leq i\leq 6$ to construct edges with over points:
\begin{align*}
i=2:& & a_2-a_3 &= 2, &\text{2 edges with 0 over points.}\\
i=3:& & a_3 - a_4 &= 0, &\text{0 edges with 1 over point.} \\
i=4:& &a_4 - a_5 &= 1, &\text{1 edge with 2 over points.}\\
i=5:& &a_5 - a_6 &= 1, &\text{1 edge with 3 over points.}\\
i=6:& &a_6 - a_7 &= 1, &\text{1 edge with 4 over points.}\\
\end{align*}

In the box \(A_1\) put the edge with 3 over points and distribute the others even edges in the boxes \(A_1\) and \(C\), respecting the restriction: \(|C|=0\) or \(|C|\ge 3\). For example, we can put the tree edges in the box \(C\) and one in the box \(A_1\). Then, glue the vertices in the box \(A_1\) to form a tree and glue the vertices of the box \(C\) to form a graph that union of unicyclic graphs. Finally, direct it taking in account that \((\kappa,G)\) must be in the family \(\mathcal{F}\). The following is an example of chain digraph that can be constructed following these steps:

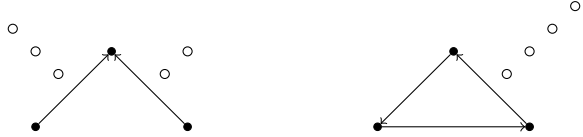
\begin{center}
\begin{tikzpicture}
\tikzset{every loop/.style={looseness=30}}

% Nodos principales
\node (1) at (0,0) [draw,circle,inner sep=1pt,fill=black!100,label=90:{}] {};
\node (2) at (1,1) [draw,circle,inner sep=1pt,fill=black!100] {};
\node (3) at (2,0) [draw,circle,inner sep=1pt,fill=black!100,label=90:{}] {};
\node (4) at (4.5,0) [draw,circle,inner sep=1pt,fill=black!100] {};
\node (5) at (5.5,1) [draw,circle,inner sep=1pt,fill=black!100] {};
\node (6) at (6.5,0) [draw,circle,inner sep=1pt,fill=black!100] {};
% cadenas w_{12}^k 
\node (7) at (0.3,0.7) [draw,circle,inner sep=1.2pt,label=0:{}] {};
\node (8) at (0,1) [draw,circle,inner sep=1.2pt,label=0:{}] {};
\node (9) at (-0.3,1.3) [draw,circle,inner sep=1.2pt,label=0:{}] {};
% cadenas w_{23}^k 
\node (10) at (1.7,0.7) [draw,circle,inner sep=1.2pt,label=0:{}] {};
\node (11) at (2,1) [draw,circle,inner sep=1.2pt,label=0:{}] {};
% cadenas w_{56}^k 
\node (12) at (6.2,0.7) [draw,circle,inner sep=1.2pt,label=0:{}] {};
\node (13) at (6.5,1) [draw,circle,inner sep=1.2pt,label=0:{}] {};
\node (14) at (6.8,1.3) [draw,circle,inner sep=1.2pt,label=0:{}] {};
\node (13) at (7.1,1.6) [draw,circle,inner sep=1.2pt,label=0:{}] {};

\path
(1) edge [->] (2)
(2) edge [<-] (3)
(4) edge [<-] (5)
(5) edge [<-] (6)
(4) edge [->] (6);

\end{tikzpicture}
\end{center}

A symplectic form of \(\g(\kappa,G)\) is given by 
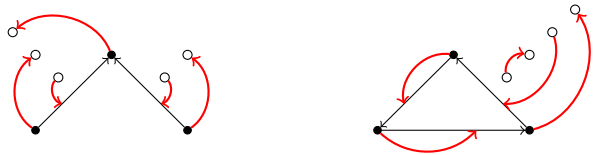
\begin{center}
\begin{tikzpicture}
\tikzset{every loop/.style={looseness=30}}

% Nodos principales
\node (1) at (0,0) [draw,circle,inner sep=1pt,fill=black!100,label=90:{}] {};
\node (2) at (1,1) [draw,circle,inner sep=1pt,fill=black!100] {};
\node (3) at (2,0) [draw,circle,inner sep=1pt,fill=black!100] {};
\node (4) at (4.5,0) [draw,circle,inner sep=1pt,fill=black!100] {};
\node (5) at (5.5,1) [draw,circle,inner sep=1pt,fill=black!100] {};
\node (6) at (6.5,0) [draw,circle,inner sep=1pt,fill=black!100] {};
% cadenas w_{12}^k 
\node (7) at (0.3,0.7) [draw,circle,inner sep=1.2pt,label=0:{}] {};
\node (8) at (0,1) [draw,circle,inner sep=1.2pt,label=0:{}] {};
\node (9) at (-0.3,1.3) [draw,circle,inner sep=1.2pt,label=0:{}] {};
% cadenas w_{23}^k 
\node (10) at (1.7,0.7) [draw,circle,inner sep=1.2pt,label=0:{}] {};
\node (11) at (2,1) [draw,circle,inner sep=1.2pt,label=0:{}] {};
% cadenas w_{56}^k 
\node (12) at (6.2,0.7) [draw,circle,inner sep=1.2pt,label=0:{}] {};
\node (13) at (6.5,1) [draw,circle,inner sep=1.2pt,label=0:{}] {};
\node (14) at (6.8,1.3) [draw,circle,inner sep=1.2pt,label=0:{}] {};
\node (15) at (7.1,1.6) [draw,circle,inner sep=1.2pt,label=0:{}] {};
%Puntos medios
\coordinate (mid_v1_v2) at ($ (1)!0.35!(2) $);
\coordinate (mid_v2_v3) at ($ (2)!0.65!(3) $);
\coordinate (mid_v4_v5) at ($ (4)!0.35!(5) $);
\coordinate (mid_v5_v6) at ($ (5)!0.65!(6) $);
\coordinate (mid_v4_v6) at ($ (4)!0.65!(6) $);

\path
(1) edge [->] (2)
(2) edge [<-] (3)
(4) edge [<-] (5)
(5) edge [<-] (6)
(4) edge [->] (6);
% --- Flechas curvas ---
\draw[->,thick,bend left=55,red] (1) to (8);
\draw[->,thick,bend right=55,red] (7) to (mid_v1_v2);
\draw[->,thick,bend right=55,red] (2) to (9);
\draw[->,thick,bend right=55,red] (3) to (11);
\draw[->,thick,bend left=55,red] (10) to (mid_v2_v3);
\draw[->,thick,bend right=55,red] (5) to (mid_v4_v5);
\draw[->,thick,bend right=45,red] (4) to (mid_v4_v6);
\draw[->,thick,bend right=55,red] (6) to (15);
\draw[->,thick,bend left=55,red] (14) to (mid_v5_v6);
\draw[->,thick,bend left=55,red] (12) to (13);

\end{tikzpicture}
\end{center}

\end{example}

\begin{example}
 Given the sequence of numbers $(15,11,7,7,4,3,1,0)$, 
 which satisfies the conditions of the previous theorem, we construct a chain digraph \((\kappa,G)\) such that the Lie algebra \(\g(\kappa,G)\) is symplectic and its type is $(15,11,7,7,4,3,1,0)$.
 
 First,
\[S=a_3+a_5+a_7-(a_4+a_6)=2 \quad \text{ and } \quad T=a_1-a_2-S=2.\]

So, there are 2 boxes of type \(A_i\), 2 boxes of type \(B_i\) and the box \(C\).

Compute the differences $a_i - a_{i+1}$ for all $2\leq i\leq 7$ to construct edges with over points:
\begin{align*}
i=2:& & a_2-a_3 &= 4, &\text{4 edges with 0 over points.}\\
i=3:& & a_3 - a_4 &= 0, &\text{0 edges with 1 over point.} \\
i=4:& &a_4 - a_5 &= 3, &\text{3 edge with 2 over points.}\\
i=5:& &a_5 - a_6 &= 1, &\text{1 edge with 3 over points.}\\
i=6:& &a_6 - a_7 &= 2, &\text{2 edge with 4 over points.}\\
i=7:& &a_7 - a_8 &= 1, &\text{1 edge with 5 over points.}
\end{align*}

In the boxes \(A_1\) and \(A_2\) put the edges with three over points and 5 over points respectively. In the boxes \(B_1\) and \(B_2\) put one vertex on each one. Distribute the others even edges in the boxes \(A_1, A_2, B_1, B_2\) and \(C\), respecting the restriction: \(|C|=0\) or \(|C|\ge 3\). Then, glue the vertices in the boxes \(A_i\) and \(B_i\) to form trees and glue the vertices in the box \(C\) to form an union of unicyclic graphs. Finally, direct it taking in account that \((\kappa,G)\) must be in the family \(\mathcal{F}\). The following is an example of chain digraph that can be constructed following these steps:

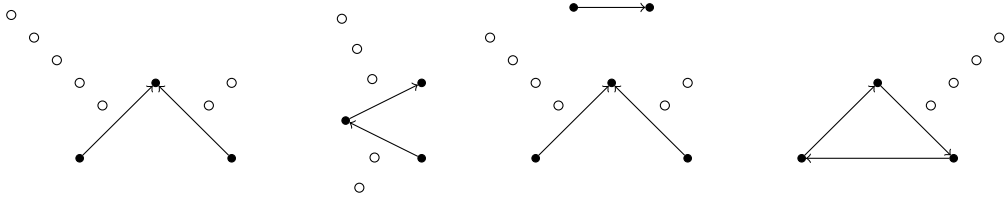
\begin{center}
\begin{tikzpicture}
\tikzset{every loop/.style={looseness=30}}
% Nodos principales
\node (1) at (0,0) [draw,circle,inner sep=1pt,fill=black!100,label=90:{}] {};
\node (2) at (1,1) [draw,circle,inner sep=1pt,fill=black!100] {};
\node (3) at (2,0) [draw,circle,inner sep=1pt,fill=black!100] {};
%2-árbol
\node (4) at (3.5,0.5) [draw,circle,inner sep=1pt,fill=black!100] {};
\node (5) at (4.5,1) [draw,circle,inner sep=1pt,fill=black!100] {};
\node (6) at (4.5,0) [draw,circle,inner sep=1pt,fill=black!100] {};
% 3árbol
\node (7) at (6,0) [draw,circle,inner sep=1pt,fill=black!100,label=90:{}] {};
\node (8) at (7,1) [draw,circle,inner sep=1pt,fill=black!100] {};
\node (9) at (8,0) [draw,circle,inner sep=1pt,fill=black!100] {};

\node (35) at (6.5,2) [draw,circle,inner sep=1pt,fill=black!100] {};
\node (36) at (7.5,2) [draw,circle,inner sep=1pt,fill=black!100] {};
%C3
\node (10) at (9.5,0) [draw,circle,inner sep=1pt,fill=black!100,label=90:{}] {};
\node (11) at (10.5,1) [draw,circle,inner sep=1pt,fill=black!100] {};
\node (12) at (11.5,0) [draw,circle,inner sep=1pt,fill=black!100] {};
% cadenas árbol1 
\node (13) at (0.3,0.7) [draw,circle,inner sep=1.2pt,label=0:{}] {};
\node (14) at (0,1) [draw,circle,inner sep=1.2pt,label=0:{}] {};
\node (15) at (-0.3,1.3) [draw,circle,inner sep=1.2pt,label=0:{}] {};
\node (16) at (-0.6,1.6) [draw,circle,inner sep=1.2pt,label=0:{}] {};
\node (17) at (-0.9,1.9) [draw,circle,inner sep=1.2pt,label=0:{}] {};

\node (18) at (1.7,0.7) [draw,circle,inner sep=1.2pt,label=0:{}] {};
\node (19) at (2,1) [draw,circle,inner sep=1.2pt,label=0:{}] {};
%cadenas árbol2
\node (20) at (3.85,1.05) [draw,circle,inner sep=1.2pt,label=0:{}] {};
\node (21) at (3.65,1.45) [draw,circle,inner sep=1.2pt,label=0:{}] {};
\node (22) at (3.45,1.85) [draw,circle,inner sep=1.2pt,label=0:{}] {};

\node (23) at (3.88,0.01) [draw,circle,inner sep=1.2pt,label=0:{}] {};
\node (24) at (3.68,-0.39) [draw,circle,inner sep=1.2pt,label=0:{}] {};
%cadenas árbol3
\node (25) at (6.3,0.7) [draw,circle,inner sep=1.2pt,label=0:{}] {};
\node (26) at (6,1) [draw,circle,inner sep=1.2pt,label=0:{}] {};
\node (27) at (5.7,1.3) [draw,circle,inner sep=1.2pt,label=0:{}] {};
\node (28) at (5.4,1.6) [draw,circle,inner sep=1.2pt,label=0:{}] {};

\node (29) at (7.7,0.7) [draw,circle,inner sep=1.2pt,label=0:{}] {};
\node (30) at (8,1) [draw,circle,inner sep=1.2pt,label=0:{}] {};
%cadenas C3
\node (31) at (11.2,0.7) [draw,circle,inner sep=1.2pt,label=0:{}] {};
\node (32) at (11.5,1) [draw,circle,inner sep=1.2pt,label=0:{}] {};
\node (33) at (11.8,1.3) [draw,circle,inner sep=1.2pt,label=0:{}] {};
\node (34) at (12.1,1.6) [draw,circle,inner sep=1.2pt,label=0:{}] {};

\path
(1) edge [->] (2)
(2) edge [<-] (3)
(4) edge [->] (5)
(4) edge [<-] (6)
(7) edge [->] (8)
(8) edge [<-] (9)
(10) edge[->] (11)
(11) edge[->] (12)
(12) edge[->] (10)
(35) edge[->] (36);
\end{tikzpicture}
\end{center}

A symplectic form of \(\g(\kappa,G)\) is given by 

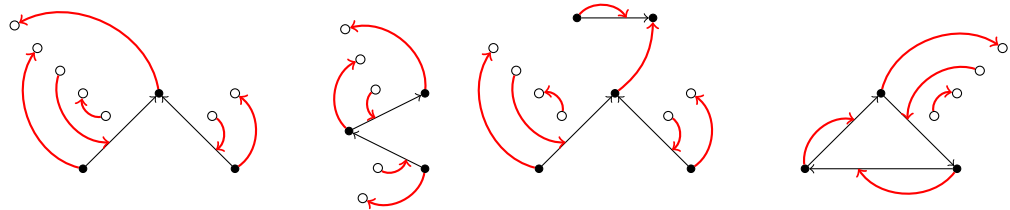
\begin{center}
\begin{tikzpicture}
\tikzset{every loop/.style={looseness=30}}

% Nodos principales
\node (1) at (0,0) [draw,circle,inner sep=1pt,fill=black!100,label=90:{}] {};
\node (2) at (1,1) [draw,circle,inner sep=1pt,fill=black!100] {};
\node (3) at (2,0) [draw,circle,inner sep=1pt,fill=black!100] {};
%2-árbol
\node (4) at (3.5,0.5) [draw,circle,inner sep=1pt,fill=black!100] {};
\node (5) at (4.5,1) [draw,circle,inner sep=1pt,fill=black!100] {};
\node (6) at (4.5,0) [draw,circle,inner sep=1pt,fill=black!100] {};
% 3árbol
\node (7) at (6,0) [draw,circle,inner sep=1pt,fill=black!100,label=90:{}] {};
\node (8) at (7,1) [draw,circle,inner sep=1pt,fill=black!100] {};
\node (9) at (8,0) [draw,circle,inner sep=1pt,fill=black!100] {};

\node (35) at (6.5,2) [draw,circle,inner sep=1pt,fill=black!100] {};
\node (36) at (7.5,2) [draw,circle,inner sep=1pt,fill=black!100] {};
%C3
\node (10) at (9.5,0) [draw,circle,inner sep=1pt,fill=black!100,label=90:{}] {};
\node (11) at (10.5,1) [draw,circle,inner sep=1pt,fill=black!100] {};
\node (12) at (11.5,0) [draw,circle,inner sep=1pt,fill=black!100] {};
% cadenas árbol1 
\node (13) at (0.3,0.7) [draw,circle,inner sep=1.2pt,label=0:{}] {};
\node (14) at (0,1) [draw,circle,inner sep=1.2pt,label=0:{}] {};
\node (15) at (-0.3,1.3) [draw,circle,inner sep=1.2pt,label=0:{}] {};
\node (16) at (-0.6,1.6) [draw,circle,inner sep=1.2pt,label=0:{}] {};
\node (17) at (-0.9,1.9) [draw,circle,inner sep=1.2pt,label=0:{}] {};

\node (18) at (1.7,0.7) [draw,circle,inner sep=1.2pt,label=0:{}] {};
\node (19) at (2,1) [draw,circle,inner sep=1.2pt,label=0:{}] {};
%cadenas árbol2
\node (20) at (3.85,1.05) [draw,circle,inner sep=1.2pt,label=0:{}] {};
\node (21) at (3.65,1.45) [draw,circle,inner sep=1.2pt,label=0:{}] {};
\node (22) at (3.45,1.85) [draw,circle,inner sep=1.2pt,label=0:{}] {};

\node (23) at (3.88,0.01) [draw,circle,inner sep=1.2pt,label=0:{}] {};
\node (24) at (3.68,-0.39) [draw,circle,inner sep=1.2pt,label=0:{}] {};
%cadenas árbol3
\node (25) at (6.3,0.7) [draw,circle,inner sep=1.2pt,label=0:{}] {};
\node (26) at (6,1) [draw,circle,inner sep=1.2pt,label=0:{}] {};
\node (27) at (5.7,1.3) [draw,circle,inner sep=1.2pt,label=0:{}] {};
\node (28) at (5.4,1.6) [draw,circle,inner sep=1.2pt,label=0:{}] {};

\node (29) at (7.7,0.7) [draw,circle,inner sep=1.2pt,label=0:{}] {};
\node (30) at (8,1) [draw,circle,inner sep=1.2pt,label=0:{}] {};
%cadenas C3
\node (31) at (11.2,0.7) [draw,circle,inner sep=1.2pt,label=0:{}] {};
\node (32) at (11.5,1) [draw,circle,inner sep=1.2pt,label=0:{}] {};
\node (33) at (11.8,1.3) [draw,circle,inner sep=1.2pt,label=0:{}] {};
\node (34) at (12.1,1.6) [draw,circle,inner sep=1.2pt,label=0:{}] {};
%´puntos medios
\coordinate (mid_v1_v2) at ($ (1)!0.35!(2) $);
\coordinate (mid_v2_v3) at ($ (2)!0.75!(3) $);
\coordinate (mid_v4_v5) at ($ (4)!0.35!(5) $);
\coordinate (mid_v4_v6) at ($ (4)!0.75!(6) $);
\coordinate (mid_v7_v8) at ($ (7)!0.35!(8) $);
\coordinate (mid_v8_v9) at ($ (8)!0.75!(9) $);
\coordinate (mid_v10_v11) at ($ (10)!0.65!(11) $);
\coordinate (mid_v11_v12) at ($ (11)!0.35!(12) $);
\coordinate (mid_v10_v12) at ($ (10)!0.35!(12) $);
\coordinate (mid_v35_v36) at ($ (35)!0.65!(36) $);

\path
(1) edge [->] (2)
(2) edge [<-] (3)
(4) edge [->] (5)
(4) edge [<-] (6)
(7) edge [->] (8)
(8) edge [<-] (9)
(10) edge[->] (11)
(11) edge[->] (12)
(12) edge[->] (10)
(35) edge[->] (36);

% --- Flechas curvas ---
\draw[->,thick,bend right=55,red] (2) to (17);
\draw[->,thick,bend left=55,red] (1) to (16);
\draw[->,thick,bend right=55,red] (15) to (mid_v1_v2);
\draw[->,thick,bend left=55,red] (13) to (14);
\draw[->,thick,bend right=55,red] (3) to (19);
\draw[->,thick,bend left=55,red] (18) to (mid_v2_v3);
\draw[->,thick,bend right=55,red] (5) to (22);
\draw[->,thick,bend left=55,red] (4) to (21);
\draw[->,thick,bend right=55,red] (20) to (mid_v4_v5);
\draw[->,thick,bend left=55,red] (6) to (24);
\draw[->,thick,bend right=55,red] (23) to (mid_v4_v6);
\draw[->,thick,bend left=55,red] (7) to (28);
\draw[->,thick,bend right=55,red] (27) to (mid_v7_v8);
\draw[->,thick,bend right=55,red] (25) to (26);
\draw[->,thick,bend right=25,red] (8) to (36);
\draw[->,thick,bend left=55,red] (35) to (mid_v35_v36);
\draw[->,thick,bend right=55,red] (9) to (30);
\draw[->,thick,bend left=55,red] (29) to (mid_v8_v9);
\draw[->,thick,bend left=55,red] (10) to (mid_v10_v11);
\draw[->,thick,bend left=55,red] (11) to (34);
\draw[->,thick,bend right=55,red] (33) to (mid_v11_v12);
\draw[->,thick,bend left=55,red] (31) to (32);
\draw[->,thick,bend left=55,red] (12) to (mid_v10_v12);

\end{tikzpicture}
\end{center}

\end{example}

\section*{Acknowledgements}

We thank Laura Barberis for useful discussion we had on the almost abelian case. 

%===============================================

 \newpage

\end{document}